\documentclass[11pt]{amsart}
\usepackage{epsfig,rlepsf}
\usepackage{amssymb}
\usepackage{amsthm}
\usepackage[all,knot,arc]{xy}
\usepackage{subfigure}
\usepackage{color}
\usepackage[breaklinks]{hyperref}
\hypersetup{backref,colorlinks=true,citecolor=blue,linkcolor=blue}
\newtheorem{theorem}{Theorem}[section]
\newtheorem{lemma}[theorem]{Lemma}
\newtheorem{proposition}[theorem]{Proposition}
\newtheorem{corollary}[theorem]{Corollary}

 \usepackage{palatino}

\theoremstyle{definition}

\newtheorem{remark}[theorem]{Remark}
\newtheorem{construction}[theorem]{Construction}

\def\dfn#1{{\em #1}}
\DeclareMathOperator{\tb}{tb}
\DeclareMathOperator{\rot}{r}
\renewcommand{\sl}{\ensuremath{{\, sl}}}
\newcommand{\N}{\ensuremath{\mathbb{N}}}

\newcommand{\R}{\ensuremath{\mathbb{R}}}
\newcommand{\Z}{\ensuremath{\mathbb{Z}}}

\newcommand{\K}{\ensuremath{\mathcal{K}}}
\renewcommand{\L}{\ensuremath{\mathcal{L}}}

\numberwithin{equation}{section}

\title[Cables of positive torus knots]{Legendrian and transverse\\ cables of positive torus knots}

\author{John B. Etnyre}
\address{School of Mathematics \\ Georgia Institute of Technology}
\email{etnyre@math.gatech.edu}
\urladdr{\href{http://www.math.gatech.edu/~etnyre}{http://www.math.gatech.edu/\~{}etnyre}}

\author{Douglas J. LaFountain}
\address{Centre for Quantum Geometry of Moduli Spaces \\ Aarhus University}
\email{dlafount@imf.au.dk}
\urladdr{\href{http://pure.au.dk/portal/en/dlafount@imf.au.dk}{http://pure.au.dk/portal/en/dlafount@imf.au.dk}}

\author{B\"{u}lent Tosun}
\address{School of Mathematics \\ Georgia Institute of Technology}
\email{btosun3@math.gatech.edu}
\urladdr{\href{http://www.math.gatech.edu/users/btosun3}{http://www.math.gatech.edu/users/btosun3}}

\begin{document}

\maketitle

\begin{abstract}
In this paper we classify Legendrian and transverse knots
in the knot types obtained from positive torus knots by cabling. This
classification allows us to demonstrate several new
phenomena. Specifically, we show there are knot types that have
non-destabilizable Legendrian representatives whose Thurston-Bennequin
invariant is arbitrarily far from maximal. We also exhibit Legendrian
knots requiring arbitrarily many stabilizations before they become
Legendrian isotopic. Similar new phenomena are observed for transverse
knots. To achieve these results we define and study ``partially
thickenable" tori, which allow us to completely classify solid tori
representing positive torus knots.
\end{abstract}

 \section{Introduction}

An $(r,s)$-curve on the boundary of a
solid torus refers to the curve $s[\lambda]+r[\mu]$, where
$\lambda,\mu$ is the longitude-meridian basis for the homology of the
torus, and we denote this by the fraction $\frac sr.$ The $(r,s)$-cable of
a knot type $\mathcal K$, denoted $\mathcal K_{(r,s)}$, is the knot type 
obtained by taking the $(r,s)$-curve on the boundary of a tubular 
neighborhood of a representative of $\mathcal{K}$. Let 
$\mathcal K$ be a positive $(p,q)$-torus knot, where we may assume
$q > p > 1$ and $\gcd{(p,q)}=1$ and let $\mathcal K_{(r,s)}$ be its
$(r,s)$-cable, also with $\gcd{(r,s)}=1$. This paper concerns the
classification of Legendrian and transverse knots representing 
$\mathcal K_{(r,s)}$ and solid tori representing $\mathcal K$. 
Though the proofs of our
classification results are heavily dependent on the ambient 
contact manifold being $(S^3, \xi_{std})$, all the Legendrian and
transversal classification results hold in any tight contact manifold, 
as can be seen by consulting \cite{EtnyreHonda03}.

Studying Legendrian and transverse knots in cabled knot types has been
very fruitful. For example, in \cite{BakerEtnyreVanHornMorris10} 
cabling was used to better understand open book decompositions of 
contact structures; in particular, leading to non-positive monodromy 
maps supporting Stein fillable contact structures, monoids in the mapping
class group associated to contact geometry and procedures to construct
open books on manifolds after allowable transverse surgery (from an 
open book for the original contact manifold). Moreover, 
the first classification of a non-transversely simple
knot type was done in \cite{EtnyreHonda05} for the $(2,3)$-cable of
the $(2,3)$-torus knot. In that paper it was also shown that studying
solid tori with convex boundary that represent a given knot type (that
is, their core curves are in a given knot type) is key to
understanding cables; such an analysis for solid tori representing
negative torus knots yielded simple Legendrian and transverse
classifications for cables of negative torus knots. Tori representing
iterated cables of torus knots were further studied in
\cite{LaFountain1, LaFountain2} as well as \cite{Tosun??}. Building on
these works we completely classify embeddings of solid
tori representing positive torus knots and use this to give a complete 
classification of Legendrian and transverse knots in the knot types of 
cables of positive torus knots.
 
Before discussing the technical classification results we state
qualitative versions that demonstrate new phenomena in the geography
of Legendrian knots. We begin with some notation. Given a topological
knot type $\K$ and integers $t$ and $r$ we denote by $\mathcal{L}(\K)$
the set of Legendrian knots (up to Legendrian isotopy) topologically
isotopic to $\K$ and by
\[ \mathcal{L}_{(r,t)}(\K)=\{L\in\mathcal{L}(\K): \tb(L)=t \text{ and
} \rot(L)=r\}.
\] We similarly denote the set of transverse knots isotopic to $\K$ by
$\mathcal{T}(\K)$ and the ones having self-linking number $s$ by
$\mathcal{T}_s(\K)$.

We first consider cables of the right handed trefoil, that is, the
$(2,3)$-torus knot.

\begin{theorem}\label{qual1} 
Let $\K$ be the positive trefoil knot in $S^3.$ The knot $\K_{(r,s)}$
formed by $(r,s)$-cabling $\K$ is Legendrian simple if and only if
$\frac sr \not\in(1,\infty).$ Furthermore, given positive integers
$k$, $m$, and $n$, where $n > 1$ and $\gcd{(k,m)}=1$, there exists a
slope $\frac sr \in (1,\infty)$ such that $\mathcal{L}_{(u,t)}(\K_{(r,
s)})$ contains $n$ Legendrian knots for some pair of integers $(u,t)$
with $t= \overline{\tb}(\K_{(r, s)})-m$; moreover, one of these does
not destabilize, and they remain distinct when stabilized fewer than
$k$ times (and there are $k$ stabilizations that will make them
isotopic).
\end{theorem}

\begin{remark} 
This theorem gives the first example of a knot type with
non-destabilizable Legendrian knots with Thurston-Bennequin invariant
arbitrarily far from the maximal Thurston-Bennequin invariant.  We
note that in \cite{EtnyreNgVertesi} it was shown there are knot types
that have arbitrarily many Legendrian knots with fixed classical
invariants, so the above theorem gives only the second family of knots
known to have this property. We also observe that this theorem gives
the first set of Legendrian knots with the same invariants that
requires arbitrarily many stabilizations before becoming Legendrian
isotopic.
\end{remark}

\begin{theorem}\label{qual2}
Let $\K$ be the positive trefoil knot in $S^3.$ The knot $\K_{(r,s)}$
formed by $(r,s)$-cabling $\K$ is transversely simple if and only if
$\frac sr \not\in(1,\infty).$ Furthermore, given positive integers
$k$, $m$, and $n$, where $n > 2$ and $\gcd{(k,m)}=1$, let
$p=k(n-1)+m(n-2)$.  Then there is some $\frac sr\in (1,\infty)$ such
that $\mathcal{T}(\K_{(r,s)})$ contains $(n-1)$ distinct transverse
knots with $sl = \overline{sl}(\K_{(r,s)}) - 2p$, of which $(n-2)$ are
non-destabilizable, and such that there is another non-destabilizable
knot with $sl = \overline{sl}(\K_{(r,s)}) - 2(p+m)$. Moreover, these
non-destabilizable knots must be stabilized until their self-linking
number is $\overline{sl}(\K_{(r,s)}) -2(p+m+k)$ before they become
transversely isotopic.
\end{theorem}

\begin{remark}
In \cite{EtnyreNgVertesi} it was
shown that there are knot types, specifically certain twist knots,
that have arbitrarily many transverse knots with the same self-linking
number. The above theorem also gives such examples but, in addition,
demonstrates three new phenomena concerning transverse
knots that were not previously known.  
Specifically it gives the first example of knot types that
have transverse knots with the same self-linking number that require
arbitrarily many stabilizations before they become transversely
isotopic, and it also gives the first examples where there are
non-destabilizable transverse knots whose self-linking number is
arbitrarily far from maximal. Finally, the theorem also gives the
first knot type where there are non-destabilizable knots with distinct
self-linking numbers.
\end{remark}

With all the interesting and complicated behavior exhibited by cables
of the right handed trefoil knot, one would expect to see behavior at
least as complicated for cables of other positive torus
knots. Surprisingly, cables of such knots turn out to be relatively
simple.

\begin{theorem}\label{qual3}
Let $\K$ be a positive $(p,q)$-torus knot with $(p,q)\not=(2,3)$. Then
for any rational number $\frac sr$ and any $(u,t)$ with $t+u$ odd,
there are at most 3 Legendrian knots in $\mathcal{L}_{(u,t)}(\K_{(r,
s)})$ and at most 2 for all but one pair $(u,t)$.
\end{theorem}

\begin{theorem}\label{qual4}
Let $\K$ be a positive $(p,q)$-torus knot with $(p,q)\not=(2,3)$. Then
for any rational number $\frac sr$ there are at most two transverse
knots isotopic to the $(r,s)$-cable of $\K$ with the same self-linking
number. However, for any positive integers $n$ and $m$ with
$\gcd{(m,n)}=1$, there is a rational number $\frac sr>0$ for which
there is a non-destabilizable transverse knot with self-linking number
at most $\overline{\sl}(\K_{(r,s)})-2n$ and it must be stabilized
exactly $m$ times to become isotopic to the destabilizable transverse
knot with the same self-linking number.
\end{theorem}

As indicated above the key to proving these classification results is
classifying solid tori with convex boundary realizing positive torus
knots. This classification, discussed below, is the first complete
such classification and exhibits features not seen before, such as the
existence of partially thickenable tori (see Subsection~\ref{Solid
tori}).

In the next two subsections we state the precise classification
theorems that lead to the above qualitative results.  In
Subsection~\ref{knotclassifications} we state knot classification
theorems for cables; in Subsection~\ref{Solid tori} we state
classification theorems for embeddings of solid tori.

\subsection{Classification results for cable knots} 
\label{knotclassifications}

We begin with cables of the right handed trefoil knot. 

\begin{theorem}\label{main23}
Let $\K$ be the $(2,3)$-torus knot. Then the $(r,s)$-cable of $\K$,
$\K_{(r,s)}$, is Legendrian simple if and only if $\frac sr \not\in
(1,\infty)$, and the classification of Legendrian knots in the knot
type $\K_{(r,s)}$ is given as follows.
\medskip
\begin{enumerate}
\item If $\frac sr \in (0,1]$ then there is a unique Legendrian knot
$L\in\mathcal{L}(\K_{(r,s)})$ with Thurston-Bennequin invariant
$\tb(L)=rs+s-r$ and rotation number $r(L)=0.$ All others are
stabilizations of $L.$
\medskip
\item If $\frac sr<0$, then the maximal Thurston-Bennequin invariant
for a Legendrian knot in $\mathcal{L}(\K_{(r,s)})$ is $rs$ and the
rotation numbers realized by Legendrian knots with this
Thurston-Bennequin invariant are
 
\[\{\pm (r + s(n+k)) \mbox{ } | \mbox{ } k=(1+n), (1+n)-2, \ldots, -(1+n) \},\]
where $n$ is the integer that satisfies
$$-n-1 < {\frac rs}< -n.$$
All other Legendrian knots $L\in\mathcal{L}(\K_{(r,s)})$ are
stabilizations of these. Two Legendrian knots with the same $\tb$ and
$\rot$ are Legendrian isotopic.
\medskip 
\item Suppose $\frac sr \in [n, n+1)$ for a positive integer $n$; then
$\K_{(r,s)}$ is not Legendrian simple and has the following
classification (see also Figure~\ref{fig:range}).
\smallskip
\begin{enumerate}
\item The maximal Thurston-Bennequin number is
$\overline{\tb}(\K_{(r,s)})=rs$.
\item There are $n$ Legendrian knots $L_\pm^j\in
\mathcal{L}(\K_{(r,s)}), j=1,\ldots, n$, with
\[
\tb(L_\pm^j)=rs\,\, \text{ and } \rot(L_\pm^j)=\pm(s-r).
\]
\item If $\frac sr \not=n$ then there are two Legendrian knots
$K_\pm\in \mathcal{L}(\K_{(r,s)})$ that do not destabilize but have
\[
\tb(K_\pm)=rs- \left|r(n+1)-s\right| \,\, \text{ and }  \rot(K_\pm)=\pm \left(s-r+ \left| r(n+1)-s\right|\right).
\]
\item All Legendrian knots in $\mathcal{L}(\K_{(r,s)})$ destabilize to
one of the $L^j_\pm$ or $K_\pm$.
\item\label{item511a} Let $c=r-1$.  For any $y\in\N \cup
\left\{0\right\}$ and $x\leq c$ the Legendrian
$S_\pm^xS_\mp^y(L^j_\pm)$ is not isotopic to a stabilization of any of
the other $L_\pm^{j}$'s the $L_\mp^j$, $K_\pm$ or $K_\mp$.
\item\label{item512a} Let $c'=r-\left|r(n+1)-s\right|-1$. For any
$y\in\N \cup \left\{0\right\}$ and $x\leq c'$ the Legendrian
$S_\pm^xS_\mp^y(K_\pm)$ is not isotopic to a stabilization of any of
the $L_\pm^{j}$'s or $K_\mp$.
\item Any two stabilizations of the $L_\pm^j$ or $K_\pm$, except those
mentioned in item~(\ref{item511a}) and~(\ref{item512a}), are
Legendrian isotopic if they have the same $\tb$ and $\rot$.
\end{enumerate}
\end{enumerate}
\end{theorem}
\begin{figure}[ht]
        \relabelbox \small  {\centerline{\epsfbox{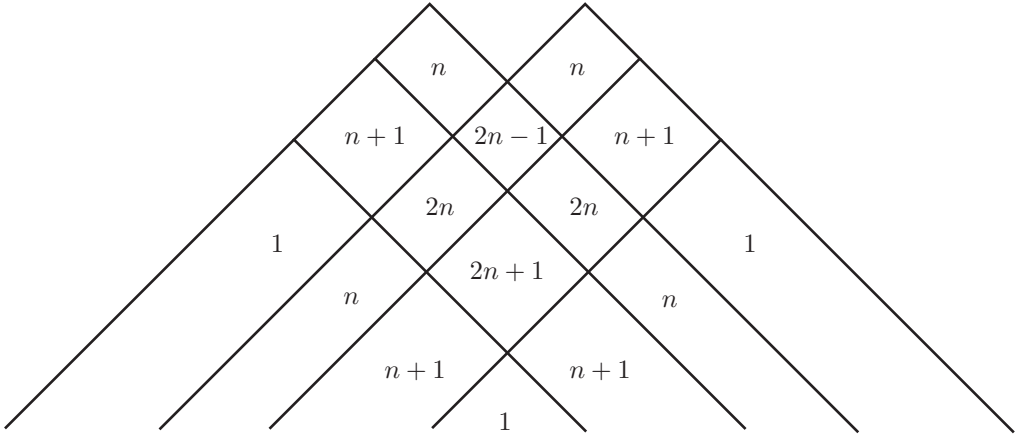}}}
  \relabel {1}{$n$}
  \relabel {2}{$n$}
  \relabel {3}{$n+1$}
    \relabel {4}{$2n-1$}
  \relabel {5}{$n+1$}
  \relabel {6}{$1$}
    \relabel {7}{$2n$}
  \relabel {8}{$2n$}
  \relabel {9}{$1$}
  \relabel{a}{$n$}
  \relabel{b}{$2n+1$}
    \relabel{c}{$n$}
  \relabel{d}{$n+1$}
    \relabel{e}{$n+1$}
  \relabel{f}{$1$}
  \endrelabelbox
        \caption{The image of $\L(\K_{(r,s)})\to \Z^2:L\mapsto
(\rot(L),\tb(L))$ for non-simple cablings of the positive trefoil with 
$\frac sr\in (n,n+1)$. The
number of Legendrian knots realizing each point in $\Z^2$ whose
coordinates sum to an odd number is indicated in the figure. The exact
width of each region is determined by Theorem~\ref{main23}.}
        \label{fig:range}
\end{figure}

The transverse classification is now an immediate corollary. 
\begin{theorem}\label{main23t}
Let $\K$ be the $(2,3)$-torus knot. If $\frac sr\not \in (1,\infty)$
then $\K_{(r,s)}$ is transversely simple and all transverse knots are
stabilizations of the one with maximal self-linking number $rs+s-r$.

If $\frac sr \in [n,n+1)$ for a positive integer $n$ then $\K_{(r,s)}$
is not transversely simple and has the following classification.
\begin{enumerate}
\item  The maximal self-linking number is $rs+s-r,$ and there is a
unique transverse knot in $\mathcal{T}(\K_{(r,s)})$ with this
self-linking number.
\item There are $n-1$ distinct transverse knots in
$\mathcal{T}(\K_{(r,s)})$ that do not destabilize and have
self-linking number $rs+r-s.$
\item If $\frac rs\not =n$ then there is a unique transverse knot in
$\mathcal{T}(\K_{(r,s)})$ that does not destabilize and has
self-linking number $rs+r-s-2|(n+1)r-s|$.
\item All other transverse knots in $\mathcal{T}(\K_{(r,s)})$
destabilize to one of the ones listed above.
\item None of the transverse knots listed above become transversely
isotopic until they have been stabilized to have self-linking number
$rs-s-r.$ There is a unique transverse knot in
$\mathcal{T}(\K_{(r,s)})$ with self-linking number less than or equal
to $rs-s-r.$
\end{enumerate}
\end{theorem}

For the classification of cables of other positive torus knots we need
some notation. Given a rational number $u=\frac sr>0$ let $u^a$ be the
largest rational number with an edge in the Farey tessellation to $u.$
See Figure~\ref{fig:nonsimple}. (The $a$ superscript stands for
"anti-clockwise", as $u^a$ is anti-clockwise of $u$ in the Farey
tessellation.) Similarly the smallest rational number with an edge in
the Farey tessellation to $u$ will be denoted by $u^c.$ A formula for
computing these numbers will be given in
Subsection~\ref{sec:continuedfrac}. We will refer to the interval
$(u^c,u^a)$ as the {\em interval of influence} for $u$.
\begin{figure}[ht]
        \relabelbox \small  {\centerline{\epsfbox{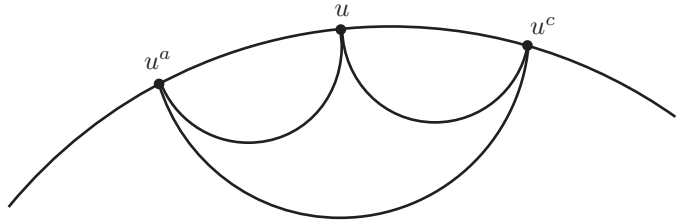}}}
  \relabel {a}{$u$}
  \relabel {b}{$u^a$}
  \relabel {c}{$u^c$}
  \endrelabelbox
        \caption{Given a rational number $u$, the numbers $u^a$ and
$u^c$ are determined by the above figure in the Farey tessellation.}
        \label{fig:nonsimple}
\end{figure}

Given a positive $(p,q)$-torus knot and $k$ a positive integer, define
\[
e_k=\frac k {pq-p-q}
\]
We will see in Subsection~\ref{Solid tori} that such $e_k$ represent
boundary slopes of {\em non-thickenable} solid tori, and that the
half-intervals of influence $(e_k,e_k^a)$ will represent boundary
slopes of {\em partially thickenable} solid tori when
$\gcd{(k,pq-p-q)}=1$.  We will refer to the $e_k$ as {\em exceptional
slopes}. If we think of the fractions $e_k^*$ as representing
curves on a torus, we denote the homological intersection of $(r,s)$
curves with the $e_k^*$ curves by
\[
\frac sr \cdot e_k^*.
\]

We can now state the precise classification theorems for cables of
general positive $(p,q)$-torus knots.

\begin{theorem}\label{main}
Let $\K$ be a $(p,q)$-torus knot with $(p,q)\not=(2,3)$. Let 
\[
\mathcal{I}=\{n\in \Z: n>1 \text{ and }  \gcd(n,pq-p-q)=1\}
\]
 and 
\[
J=\cup_{n\in \mathcal{I}} J_n
\] 
where $J_n=(e_n^c, e_n^a)$ is the interval of influence for the
exceptional slope $e_n$ defined above. The $J_n$ are all disjoint.

The classification of Legendrian knots in the knot type $\K_{(r,s)}$
is then given as follows.
\medskip 
\begin{enumerate}
\item If $\frac sr\not \in J$ then $\K_{(r,s)}$ is Legendrian
simple. Moreover, in this case we have the following classification.
\smallskip 
\begin{enumerate}
\item If $\frac sr \in (0,\frac{1}{pq-p-q}]$ then there is a unique
Legendrian knot $L\in\mathcal{L}(\K_{(r,s)})$ with Thurston-Bennequin
invariant $\tb(L)=rs+s(pq-p-q)-r$ and rotation number $r(L)=0.$ All
others are stabilizations of $L.$
\item If $\frac sr>\frac{1}{pq-p-q}$ or $\frac sr<0$, then the maximal
Thurston-Bennequin invariant for a Legendrian knot in
$\mathcal{L}(\K_{(r,s)})$ is $rs$ and the rotation numbers realized by
Legendrian knots with this Thurston-Bennequin invariant are
\[
\{\pm (r + s(n+k)) \mbox{ } | \mbox{ } k=(pq-p-q-n), (pq-p-q-n)-2, \ldots, -(pq-p-q-n) \},
\]
where $n$ is the least integer bigger than $\frac rs$.  All other
Legendrian knots $L\in\mathcal{L}(\K_{(r,s)})$ are stabilizations of
these. Two Legendrian knots with the same $\tb$ and $\rot$ are
Legendrian isotopic.
\end{enumerate}
\medskip
\item If $\frac sr\in J_n$ then there is some $k\geq 0$ such that
$\frac 1{k-1}> \frac sr >\frac 1{k}$ and $\K_{(r,s)}$ is not
Legendrian simple. The classification of Legendrian knots in
$\K_{(r,s)}$ is as follows.
\smallskip
\begin{enumerate}
\item The maximal Thurston-Bennequin invariant of $K_{(r,s)}$ is $rs$.
\item For each integer $i$ in the set 
\[
\{\pm (r + s(-k+l)) \mbox{ } | \mbox{ } l=(pq-p-q-k), (pq-p-q-k)-2,\ldots, -(pq-p-q-k) \},
\]
there is a Legendrian $L_i\in \mathcal{L}(\K_{(r,s)})$ with 
\[
\tb(L_i)=rs\,\, \text{ and } \rot(L_i)=i.
\]
\item There are two Legendrian knots $K_\pm\in
\mathcal{L}(\K_{(r,s)})$ satisfying
\[\tb(K_\pm)=rs \,\, \text{ and }  \rot(K_\pm)=\pm(r- s(pq-p-q))\]
if $\frac sr\in [e_n,e_n^a)$; however, if $\frac sr\in(e_n^c, e_n)$ then
\[
\tb(K_\pm)=rs- \left|\frac sr \cdot e_n\right| \,\, \text{ and }  \rot(K_\pm)=\pm r(n-1)
\]
and $K_\pm$ is not destabilizable.
\item All Legendrian knots in $\mathcal{L}(\K_{(r,s)})$ destabilize to
one of the $L_i$ or $K_\pm.$
\item\label{item??} Let 
\[
c=\begin{cases}(\frac sr \cdot e_n^a)-1& \frac sr\in [e_n,e_n^a)\\
(\frac sr\cdot e_n^a-\frac sr\cdot e_n)-1& \frac sr\in (e_n^c, e_n).
\end{cases}
\]
For any $y\in\N \cup \left\{0\right\}$ and $x\leq c$ the Legendrian
$S_\pm^xS_\mp^y(K_\pm)$ is not isotopic to a stabilization of any of
the $L_i$ or $K_\mp$.
\item Any two stabilizations of the non-destabilizable
Thurston-Bennequin invariant Legendrian knots in
$\mathcal{L}(\K_{(r,s)})$, except those mentioned in
item~(\ref{item??}), are Legendrian isotopic if they have the same
$\tb$ and $\rot$.
\end{enumerate}
\end{enumerate}
\end{theorem}
From this theorem we can easily derive the transverse classification.
\begin{theorem}\label{maint}
Let $\K$ be a $(p,q)$-torus knot with $(p,q)\not=(2,3)$. Using
notation from Theorem~\ref{main} we have the following classification
of transverse knots in $\mathcal{T}(\K_{(r,s)})$.
\smallskip
\begin{enumerate}
\item If $\frac sr\not \in J_n$ for any $n\in \mathcal{I}$ then $\K_{(r,s)}$ is transversely
simple and all transverse knots in this knot type are stabilizations
of the one with self-linking number $rs-r+s(pq-p-q)$.
\item If $\frac sr\in J_n$ for some $n\in \mathcal{I}$ then
$\K_{(r,s)}$ is not transversely simple. There is a unique transverse
knot $T$ in this knot type with maximal self-linking number, which is
$rs-r+s(pq-p-q)$. There is also a unique non-destabilizable knot $T'$
in this knot type and it has self-linking number $rs+r-s(pq-p-q)$. All
other transverse knots in $\mathcal{T}(\K_{(r,s)})$ destabilize to
either $T$ or $T'$ and the stabilizations of $T$ and $T'$ stay
non-isotopic until they are stabilized to the point that their
self-linking numbers are
\[
rs+r-s(pq-p-q)-2\left(\frac sr \cdot e_n^a\right)
\]
in the case of $\frac sr \in [e_n,e_n^a)$, and
\[
rs+r-s(pq-p-q)-2\left(\frac sr \cdot e_n^a - \frac sr \cdot e_n\right)
\]
in the case of $\frac sr \in (e_n^c,e_n)$.

\end{enumerate}
\end{theorem}

We now turn from classification results for cables of positive torus
knots, to classification results for embeddings of solid tori
representing the positive torus knots themselves.

 \subsection{Classification results for solid tori}
 \label{Solid tori}

Let $S$ be a solid torus in a manifold $M.$ We say $S$ \dfn{is in the
knot type} $\K$, or \dfn{represents} $\K$, if the core curve of $S$ is
in the knot type $\K.$
 
We say a solid torus $S$ with convex boundary in a contact manifold
$(M,\xi)$ \dfn{thickens} if there is a solid torus $S'$ that contains
$S$ such that $S'$ has convex boundary with dividing slope different
from $S.$ The existence of non-thickenable tori was first observed in
\cite{EtnyreHonda05}; the following theorem shows that non-thickenable
tori exist for all positive $(p,q)$-torus knots.

\begin{theorem}\label{thm:nonthick}
Let $S$ be a solid torus in the knot type of a positive $(p,q)$-torus
knot. In the standard tight contact structure $\xi_{std}$ on $S^3$
suppose that $\partial S$ is convex with two dividing curves of slope
$\frac sr.$ Then $S$ thickens unless $\frac sr$ is an exceptional
slope
 \[
e_k=\frac{k}{pq-p-q},
\] 
for some positive integer $k,$ in which case it might or might not
thicken.

Moreover for each positive integer $k>1$ there are, up to contact
isotopy, exactly two solid tori $N_k^\pm$ with convex boundary having
$2n_k$ dividing curves of slope $e_k$ that do not thicken, where
$n_k=\gcd(pq-p-q, k)$. For $k=1$ there is exactly one solid torus
$N_1$ with convex boundary having two dividing curves of slope $e_1.$
This solid torus is a standard neighborhood of a Legendrian
$(p,q)$-torus knots with maximal Thurston-Bennequin invariant and it
does not thicken.
 \end{theorem}

A key feature in the knot classification results above in
Subsection~\ref{knotclassifications} is a complete understanding of
not only non-thickenable tori but also \dfn{partially thickenable}
tori, that is tori with convex boundary that thicken, but not to a
maximally thick torus in the given knot type. The existence of such
tori has not been observed before, but it is clear that such tori will
be key to future Legendrian classification results. In addition it is
likely they will be important in understanding contact surgeries.  The
following theorem shows that partially thickenable tori exist for all
positive $(p,q)$-torus knots.

\begin{theorem}\label{thm:partially}
Let $\K$ be a positive $(p,q)$-torus knot and let
$e_k=\frac{k}{pq-p-q}$ be the exceptional slopes. Let $I_k=[e_k,
e_k^a)$ and $\mathcal{I}=\{n\in \Z: n>1 \text{ and }
\gcd(n,pq-p-q)=1\}$.  All solid tori below will represent the knot
type $\K$.
\medskip
\begin{enumerate}
\item If $(p,q)=(2,3)$ then $\mathcal{I}=\N-\left\{1\right\}$ and we
have the following.
\smallskip
\begin{enumerate}
\item The intervals $I_k=(k,\infty),$ so $I_k\subset I_{k+1}.$
\item Any solid torus $S$ with convex boundary thickens to $N_k^\pm$
or to $N_1$ (that is a neighborhood of the maximal Thurston-Bennequin
invariant $(2,3)$-torus knot).
\item Any solid torus inside $N_k^\pm$ with convex boundary having
dividing slope greater than $k$ (that is in $I_k$) does not thicken
past the slope $e_k.$
\item Any solid torus inside $N_k^\pm$ with convex boundary having
negative (or infinite) dividing slope will thicken to a neighborhood
of the maximal Thurston-Bennequin invariant $(2,3)$-torus knot.
\end{enumerate}
\medskip
\item If $(p,q)\not=(2,3)$ then we have the following.
\smallskip
\begin{enumerate}
\item For any $k\not\in \mathcal{I}$, any solid torus $S$ inside
$N_k^\pm$ with either boundary slope different from $e_k$, or less
than $2n_k$ dividing curves, thickens to $N_1$.
\item All the $I_k$ with $k\in \mathcal{I}$ are disjoint.
\item Any solid torus $S$ with convex boundary having dividing slope
in $I_k$ thickens to $N_k^\pm$ or to $N_1$ (that is a neighborhood of
the maximal Thurston-Bennequin invariant $(p,q)$-torus knot).
\item Any solid torus inside $N_k^\pm$ for some $k\in \mathcal{I}$,
and with convex boundary having dividing slope in $I_k$, does not
thicken past the slope $e_k.$
\item Any solid torus inside $N_k^\pm$ with convex boundary having
dividing slope outside of $I_k$ (that is greater than  or equal to $e_k^a$ 
or negative) will thicken to a neighborhood of the maximal
Thurston-Bennequin invariant $(p,q)$-torus knot.
\end{enumerate}
\end{enumerate}
\end{theorem}

From this theorem we can classify solid tori in the knot types of
positive torus knots.

\begin{corollary}\label{cor:tori}
Let $\K$ be a positive $(p,q)$-torus knot and let
$e_k=\frac{k}{pq-p-q}$ be the exceptional slopes. Let $I_k=[e_k,
e_k^a)$ and $\mathcal{I}=\{n\in \Z: n>1 \text{ and }
\gcd(n,pq-p-q)=1\}$.
\medskip  
\begin{enumerate}
\item If $(p,q)=(2,3)$, then given a slope $s > 1$ there is some
integer $n$ such that $n\leq s<n+1$ and there are exactly $2n$ solid
tori representing the knot type $\K$ with convex boundary having
dividing slope $s$ and two dividing curves, only two of which thicken
to a standard neighborhood of a Legendrian knot.
\medskip 
\item If $(p,q)\not=(2,3)$, then given any slope $s\geq \frac
1{pq-p-q}$ we have the following.
\smallskip 
\begin{enumerate}
\item If there is some integer $n>0$ such that $\frac 1n< s< \frac
1{n-1}$ and $s\not\in I_k$ for any $k\in \mathcal{I}$, then there are
exactly $2(pq-p-q-n+1)$ solid tori representing the knot type $\K$
with convex boundary having dividing slope $s$ and two dividing curves
each of which thickens to a standard neighborhood of a Legendrian knot
with $\tb=n$.
\item If there is some integer $n>0$ such that $\frac 1n< s< \frac
1{n-1}$ and $s\in I_k$ for any $k\in \mathcal{I}$, then there are
exactly $2(pq-p-q-n+1)+2$ solid tori representing the knot type $\K$
with convex boundary having dividing slope $s$ and two dividing
curves, all but two of which thicken to a standard neighborhood of a
Legendrian knot with $\tb=n$.
\item If there is some $n>0$ such that $s=\frac 1n$, then there are
exactly $pq-p-q-n+1$ solid tori representing the knot type $\K$ with
convex boundary having dividing slope $s$ and two dividing curves and
they each represent a standard neighborhood of a Legendrian knot with
$\tb=n$.
\end{enumerate}
\medskip
\item Given any negative slope $s$ there is some negative integer $n <
0$ such that $\frac 1{n+1} <s< \frac 1{n}$. A solid torus with convex
boundary having dividing slope $s$ and two dividing curves will
thicken to a solid torus that is a standard neighborhood of a
$\tb=n+1$ Legendrian knot.
\end{enumerate}
\end{corollary}

We conclude this introduction with an outline of what follows.  In
Section~\ref{sec:pre} we collect needed preliminaries, including facts
about continued fractions and convex surfaces, and we outline a
strategy for classifying Legendrian knots.  In
Section~\ref{Classification of solid tori} we classify embeddings
of solid tori representing positive torus knots.  In
Section~\ref{simple cables} we provide classifications for all simple
cables of positive torus knots, and in Sections~\ref{non-simple cables
non-trefoil} and ~\ref{non-simple cables trefoil} we establish
classifications for all non-simple cables of positive torus knots.

\bigskip
\noindent
{\bf Acknowledgments.}  The first and third authors were partially supported by
NSF Grant DMS-0804820. The second author was partially supported by
QGM (Centre for Quantum Geometry of Moduli Spaces) funded by the
Danish National Research Foundation. Some of the work presented in
this paper was carried out in the Spring of 2010 while the first and
third author were at MSRI, we gratefully acknowledge their support for
this work.

\section{Preliminaries}\label{sec:pre}

In this section we first prove some important facts about continued
fractions in Subsection~\ref{sec:continuedfrac}. The remaining
sections recall various facts concerning the classification of
Legendrian and transverse knots from \cite{EtnyreHonda01b}. The reader
is assumed to be familiar with the basic notions associated to convex
surfaces and Legendrian and transverse knots, but these sections are
included for the convenience of the reader and to make the paper as
self-contained as possible. All this information can be found in
\cite{EtnyreHonda01b, Honda00a}.

\subsection{Continued fractions, the Farey tessellation, and intersection of curves on a torus}\label{sec:continuedfrac}
In this section we collect various facts about continued fractions and
the Farey tessellation (see Figure~\ref{fig:ft}) that will be needed
throughout our work.

Given a rational number $u>0$ we may represent it as a continued
fraction
\[
u=a_0 - \cfrac 1{a_1- \cfrac 1{a_2 \ldots -\cfrac 1{a_n}}}
\]
with $a_0\geq 1$ and the other $a_i>1$. We will denote this as
$u=[a_0; a_1,\ldots, a_n]$. If we know that $u=[a_0;a_1,\ldots, a_n]$
then we define
\[
u^a= [a_0; a_1,\ldots, a_{n-1}],
\]
with the convention that if $n=0$ then $u^a=\infty$; we also define
\[
u^c=[a_0;a_1,\ldots, a_n-1].
\]
\begin{lemma}\label{lem:eaec}
The number $u^a$ is the largest rational number bigger than $u$ with
an edge to $u$ in the Farey tessellation and $u^c$ is the smallest
rational number less than $u$ with an edge to $u$ in the Farey
tessellation. Moreover there is an edge in the Farey tessellation
between $u^a$ and $u^c$ and $u$ is the mediant of $u^a$ and $u^c$,
that is if $u^a=\frac{p^a}{q^a}$ and $u^c=\frac{p^c}{q^c}$ then
\[
u=\frac{p^a+p^c}{q^a+q^c}.
\] 
\end{lemma}

\begin{proof}
Define $\frac{p_k}{q_k}=[a_0;a_1,\ldots, a_k]$ and $p_{-1}=1,
q_{-1}=0.$ One may easily verify using induction that
\[
p_{k+1}=a_{k+1}p_k-p_{k-1}, \text{ and }\,  q_{k+1}=a_{k+1}q_k-q_{k-1}.
\]
From this one can inductively deduce that
\[
p_{k+1}q_k-p_kq_{k+1}=-1.
\]
Thus there is an edge in the Farey tessellation between $u=\frac
pq=\frac {p_n}{q_n}$ and
$u^a=\frac{p^a}{q^a}=\frac{p_{n-1}}{q_{n-1}}.$ Similarly, let
$\frac{c_k}{d_k}=[a_k;a_{k+1},\ldots, a_n]$ and
$\frac{c'_k}{d'_k}=[a_k;a_{k+1},\ldots, a_n-1]$ and notice that
$c_nd'_n-d_nc'_n=a_n-(a_n-1)=1.$ Now we see that
\[
\frac{c_k}{d_k}=a_k-\frac{1}{c_{k+1}/{d_{k+1}}}= \frac{a_{k}c_{k+1}-d_{k+1}}{c_{k+1}}
\]
and a similar expression for $\frac{c'_k}{d'_k}$ and induction yield
$c_kd'_k-d_kc'_k=1.$ In particular, there is an edge in the Farey
tessellation between $u=\frac{c_0}{d_0}$ and $u^c=\frac{c'_0}{d'_0}.$

Finally by setting $\frac{c''_k}{d''_k}=[a_k; a_{k+1},\ldots,
a_{n-1}]$ and noticing that $c''_{n-1}d'_{n-1}-d''_{n-1}c'_{n-1}=1$,
we can use the above formulas, and analogous ones, to inductively
prove that $c''_{k-1}d'_{k-1}-d''_{k-1}c'_{k-1}=1$. This establishes
an edge in the Farey tessellation between $u^c=\frac{c'_0}{d'_0}$ and
$u^a=\frac{c''_0}{d''_0}$. Since there is an edge in the Farey
tessellation between each pair of numbers in the set $\{u, u^a, u^c\}$
the lemma is established by noticing that the numerators (and
denominators) of $u^a$ and $u^c$ are both smaller than the numerator
(and denominator) of $u.$
\end{proof}

We recall that if we choose a basis for $H_1(T^2;\Z)$ then there is a
one-to-one correspondence between embedded essential oriented curves
on $T^2$ and rational numbers $\frac pq$, written in lowest common
terms. Moreover given two rational numbers $\frac pq$ and $\frac rs$
we denote their homological intersection (which also happens to be the
signed minimal intersection number) between the corresponding curves
on $T^2$ by $\frac pq \cdot \frac rs$ and it can be computed by
\[
\frac pq \cdot \frac rs= ps-rq.
\]
Notice that this number is only well defined up to sign (since the
orientation on the curve corresponding to a fraction is not
determined). Throughout this work we will only be concerned with the
absolute value of this number (if the exact number is ever needed we
will specify the orientations on the homology class corresponding to a
fraction).

\begin{lemma}\label{lem:basicrange}
Fix some positive integer $n$ and set $e_k=\frac k n$ for $k \in
\left\{1,2, \ldots\right\}$ and $\mathcal{I}=\{k\in \Z: k>1 \text{ and
} \gcd(n,k)=1\}$.  If $n\not=1$ then the intervals $J_k=(e_k^c,
e_k^a)$ for $k\in \mathcal{I}$ are all disjoint. If $n=1$ then the
intervals are nested $J_{k+1}\subset J_k$.

If $r$ is a positive rational number less than $e_k^c$ or greater than
$e_k^a$ then for any $s\in \overline{J_k}$ we have
\[
|r\cdot s|\geq \min\{|r\cdot e_k^a|, |r\cdot e_k^c|\}
\]
with equality only if $s=e_k^a$ or $e_k^c$. 

If $r\in (e_k^c,e_k)$ and $s\in (e_k, e_k^a),$ then 
\[
|r\cdot s| > |r\cdot e_k^a|.
\]
\end{lemma}
\begin{proof}
If $n=1$ then it is clear that $e_k=k$ and one easily checks that
$e_k^c=k-1$ and $e_k^a=\infty.$ So $J_k=(k-1, \infty)$.

If $n\not=1$ then we notice that any number in $(e^c_k, e_k)$ is a
mediant of $e_k$ and $e_k^c$ and hence has denominator strictly bigger
than $n$ (since the denominator of $e_k$ is $n$), thus $e_{k'}$ cannot
be in this interval for any $k'\in \mathcal{I}$. Similarly $e_{k'}$
cannot be in the interval $(e_k, e_k^a)$. Thus the intervals $J_k,
i\in \mathcal{I}$ are disjoint.

For the second statement notice that $r\cdot e_k^a$ and $r\cdot e_k^c$
have the same sign and $r\cdot s$ will be some non-negative integral
linear combination of $r\cdot e_k^a$ and $r\cdot e_k^c$.  For the last
statement note that $r\cdot s$ will be some positive integral linear
combination of $r\cdot e_k$ and $r\cdot e_k^a$.
\end{proof}

\subsection{Convex surfaces and bypasses}\label{sec:cvx}
In this subsection we discuss the main tools we will be using
throughout the paper --- convex surfaces. We assume the reader is
familiar with convex surfaces as used in \cite{EtnyreHonda01b,
Honda00a}; but, for the convenience of the reader, we recall the 
fundamental facts from the theory that we will use in this paper.

\subsubsection{Convex surfaces}
Recall a surface $\Sigma$ in a contact manifold $(M,\xi)$ is
\dfn{convex} if it has a neighborhood $\Sigma\times I$, where
$I=(-\epsilon, \epsilon)$ is some interval, and $\xi$ is $I$-invariant
in this neighborhood. Any closed surface can be $C^\infty$-perturbed
to be convex. Moreover if $L$ is a Legendrian knot on $\Sigma$ for
which the contact framing is non-positive with respect to the framing
given by $\Sigma$, then $\Sigma$ may be perturbed in a $C^0$ fashion
near $L$, but fixing $L$, and then again in a $C^\infty$ fashion away
from $L$ so that $\Sigma$ is convex.

Given a convex surface $\Sigma$ with $I$-invariant neighborhood let
$\Gamma_\Sigma \subset \Sigma$ be the multicurve where $\xi$ is
tangent to the $I$ factor. This is called the \dfn{dividing set} of
$\Sigma.$ If $\Sigma$ is oriented it is easy to see that
$\Sigma\setminus \Gamma=\Sigma_+\cup \Sigma_-$ where $\xi$ is
positively transverse to the $I$ factor along $\Sigma_+$ and
negatively transverse along $\Sigma_-$. If $L$ is a Legendrian curve
on a $\Sigma$ then the framing of $L$ given by the contact planes,
relative to the framing coming from $\Sigma$, is given by $-\frac 12
(L\cdot \Gamma)$. Moreover if $L=\partial \Sigma$ then the rotation
number of $L$ is given by $\rot(L)=\chi(\Sigma_+)-\chi(\Sigma_-)$.

\subsubsection{Convex tori}
A convex torus $T$ is said to be in \dfn{standard form} if $T$ can be
identified with $\R^2/\Z^2$ so that $\Gamma_T$ consists of $2n$
vertical curves (note $\Gamma_T$ will always have an even number of
curves and we can choose a parameterization to make them vertical) and
the characteristic foliations consists of $2n$ vertical lines of
singularities ($n$ lines of sources and $n$ lines of sinks) and the
rest of the foliation is by non-singular lines of slope $s$. See
Figure~\ref{fig:ct}.
\begin{figure}[ht]
        \relabelbox \small  {\centerline{\epsfbox{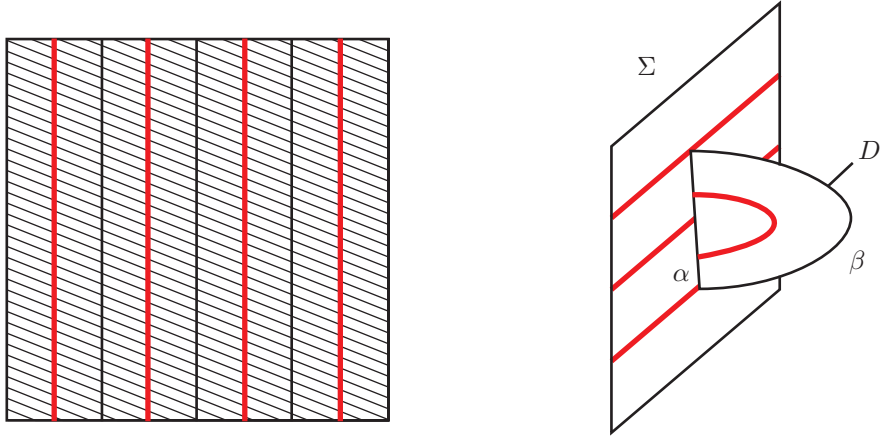}}}
          \relabel{D}{$D$}
  \relabel {S}{$\Sigma$}
  \relabel {a}{$\alpha$}
  \relabel {b}{$\beta$}
  \endrelabelbox
        \caption{Standard convex tori shown on the left and a bypass
shown on the right. The thicker curves are dividing curves.}
        \label{fig:ct}
\end{figure}
The lines of singularities are called \dfn{Legendrian divides} and the
other curves are called \dfn{ruling curves}. We notice that the Giroux
Flexibility Theorem allows us to isotope any convex torus into
standard form, \cite{EtnyreHonda01b, Honda00a}.

\subsubsection{Bypasses and tori}
Let $\Sigma$ be a convex surface and $\alpha$ a Legendrian arc in
$\Sigma$ that intersects the dividing curves $\Gamma_\Sigma$ in 3
points $p_1,p_2,p_3$ (where $p_1, p_3$ are the end points of the
arc). Then a \dfn{bypass for $\Sigma$ (along $\alpha$)}, see
Figure~\ref{fig:ct}, is a convex disk $D$ with Legendrian boundary
such that
\begin{enumerate}
\item $D\cap \Sigma=\alpha,$
\item $tb(\partial D)=-1,$
\item $\partial D= \alpha\cup \beta,$
\item $\alpha\cap\beta=\{p_1,p_3\}$ are corners of $D$ and elliptic
singularities of $D_\xi.$
\end{enumerate}
A surface $\Sigma$ locally separates the ambient manifold. If a bypass
is contained in the (local) piece of $M\setminus \Sigma$ that has
$\Sigma$ as its oriented boundary then we say the bypass will be
attached to the front of $\Sigma$ otherwise we say it is attached to
the back of $\Sigma$.

When a bypass is attached to a torus $T$ then either the dividing
curves do not change, their number increases by two, or decreases by
two, or the slope of the dividing curves changes. The slope of the
dividing curves can change only when there are two dividing
curves. (See \cite{Honda00a} for more details.) If the bypass is
attached to $T$ along a ruling curve then either the number of
dividing curves decreases by two or the slope of the dividing curves
changes. To understand the change in slope we need the following.  Let
$\mathbb{D}$ be the unit disk in $\R^2.$ Recall the \dfn{Farey
  tessellation} of $\mathbb{D}$ is constructed as follows.  Label the
point $(1,0)$ on $\partial \mathbb{D}$ by $0=\frac01$ and the point
$(-1,0)$ with $\infty=\frac10.$ Now join them by a geodesic.  If two
points $\frac{p}{q},$ $\frac{p'}{q'}$ on $\partial \mathbb{D}$ with
non-negative $y$-coordinate have been labeled then label the point on
$\partial\mathbb{D}$ half way between them (with non-negative
$y$-coordinate) by $\frac{p+p'}{q+q'}.$ Then connect this point to
$\frac{p}{q}$ by a geodesic and to $\frac{p'}{q'}$ by a
geodesic. Continue this until all positive fractions have been
assigned to points on $\partial \mathbb{D}$ with non-negative
$y$-coordinates. Now repeat this process for the points on $\partial
\mathbb{D}$ with non-positive $y$-coordinate except start with
$\infty=\frac{-1}{0}.$ See Figure~\ref{fig:ft}.
\begin{figure}[ht]
  \relabelbox \small {\epsfxsize=2.8in\centerline{\epsfbox{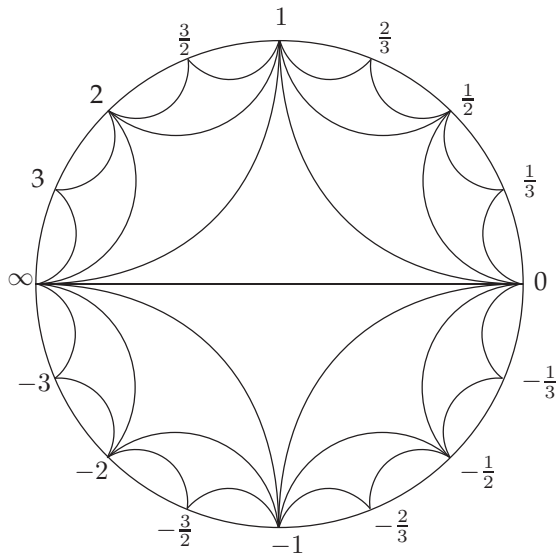}}} 
  \relabel {0}{0}
  \relabel {i}{$\infty$}
  \relabel {1}{$1$}
  \relabel {2}{2}
  \relabel {3}{3}
  \relabel {-1}{$-1$}
  \relabel {-2}{$-2$}
  \relabel {-3}{$-3$}
  \relabel {12}{$\frac12$}
  \relabel {13}{$\frac13$}
  \relabel {1}{}
  \relabel {-12}{$-\frac12$}
  \relabel {-13}{$-\frac13$}
  \relabel {23}{$\frac23$}
  \relabel {-23}{$-\frac23$}
  \relabel {32}{$\frac32$}
  \relabel {-32}{$-\frac32$}
  \endrelabelbox
        \caption{The Farey tessellation.}
        \label{fig:ft}
\end{figure}

The key result we need to know about the Farey tessellation is given
in the following theorem.
\begin{theorem}[Honda 2000, \cite{Honda00a}]\label{bpaferry}
Let $T$ be a convex torus in standard form with $|\Gamma_{T}|=2,$
dividing slope $s$ and ruling slope $r\not=s.$ Let $D$ be a bypass for
$T$ attached to the front of $T$ along a ruling curve. Let $T'$ be the
torus obtained from $T$ by attaching the bypass $D.$ Then
$|\Gamma_{T'}|=2$ and the dividing slope $s'$ of $\Gamma_{T'}$ is
determined as follows: let $[r,s]$ be the arc on $\partial\mathbb{D}$
running from $r$ counterclockwise to $s,$ then $s'$ is the point in
$[r,s]$ closest to $r$ with an edge to $s.$

If the bypass is attached to the back of $T$ then the same algorithm
works except one uses the interval $[s,r]$ on
$\partial\mathbb{D}$. \qed
\end{theorem}

\subsubsection{The Imbalance Principle}
As we see that bypasses are useful in changing dividing curves on a
surface we mention a standard way to try to find them called the
Imbalance Principle. Suppose that $\Sigma$ and $\Sigma'$ are two
disjoint convex surfaces and $A$ is a convex annulus whose interior is
disjoint from $\Sigma$ and $\Sigma'$ but its boundary is Legendrian
with one component on each surface. If $|\Gamma_\Sigma\cdot \partial
A|>|\Gamma_{\Sigma'}\cdot A|$ then there will be a dividing curve on
$A$ that cuts a disk off of $A$ that has part of its boundary on
$\Sigma$. It is now easy to use the Giroux Flexibility Theorem to show
that there is a bypass for $\Gamma$ on $A$.

\subsubsection{Discretization of Isotopy}\label{sssec:discritize}
We will frequently need to analyze what happens to the contact
geometry when we have a topological isotopy between two convex
surfaces $\Sigma$ and $\Sigma'$. This can be done by the technique of
\dfn{Isotopy Discretization} \cite{Colin97} (see also
\cite{EtnyreHonda01b} for its use in studying Legendrian knots). Given
an isotopy between $\Sigma$ and $\Sigma'$ one can find a sequence of
convex surfaces $\Sigma_1=\Sigma, \Sigma_2, \ldots, \Sigma_n=\Sigma'$
such that
\begin{enumerate}
\item all the $\Sigma_i$ are convex and
\item $\Sigma_i$ and $\Sigma_{i+1}$ are disjoint and $\Sigma_{i+1}$ is
obtained from $\Sigma_i$ by a bypass attachment.
\end{enumerate}
Thus if one is trying to understand how the contact geometry of
$M\setminus \Sigma$ and $M\setminus \Sigma'$ relate, one just needs to
analyze how the contact geometry of the pieces of $M\setminus
\Sigma_i$ changes under bypass attachment. In particular, many
arguments can be reduced from understanding a general isotopy to
understanding an isotopy between two surfaces that cobound a product
region.

There is also a relative version of Isotopy Discretization where
$\Sigma$ and $\Sigma'$ are convex surfaces with Legendrian boundary
consisting of ruling curves on a convex torus. If $\partial
\Sigma=\partial \Sigma'$ and there is a topological isotopy of
$\Sigma$ to $\Sigma'$ relative to the boundary then we can find a
discrete isotopy as described above.

\subsection{Classifying knots in a knot type}\label{sec:strategy}
In this section we briefly recall the standard strategy for
classifying Legendrian knots in a given knot type $\K$ as laid out in
\cite{Etnyre99, EtnyreHonda01b}. We begin by recalling the ``normal
form'' for a neighborhood of a Legendrian or transverse knot and the
relation between them.

\subsubsection{Standard neighborhoods of knots}

Given a Legendrian knot $L$, a \dfn{standard neighborhood} of $L$ is a
solid torus $N$ that has convex boundary with two dividing curves of
slope $1/\tb(L)$ (and of course we will usually take $\partial N$ to
be a convex torus in standard form). Conversely given any such solid
torus it is a standard neighborhood of a unique Legendrian knot. Up to
contactomorphism one can model a standard neighborhood as a
neighborhood $N'$ of the $x$-axis in $\R^3/(x\mapsto x+1)\cong
S^1\times \R^2$ with contact structure $\xi_{std}=\ker(dz-y\, dx).$
Using this model we can see that $L_\pm=\{(x, \pm\epsilon, 0)\}\subset
N'$ is a $(\pm)$-transverse curve. The image of $L_+$ in $N$ is called
the \dfn{transverse push-off} of $L$ and $L_-$ is called the
\dfn{negative transverse push-off}. One may easily check that $L_\pm$
is well-defined and compute that
\[
\sl(L_\pm)=\tb(L)\mp \rot(L).
\]

One may understand stabilizations and destabilizations of a Legendrian
knot $K$ in terms of the standard neighborhood. Specifically, inside
the standard neighborhood $N$ of $L$, $L$ can be positively stabilized
to $S_+(L)$, or negatively stabilized to $S_-(L)$. Let $N_\pm$ be a
neighborhood of the stabilization of $L$ inside $N.$ As above we can
assume that $N_\pm$ has convex boundary in standard form. It will have
dividing slope $\frac{1}{\tb(L)-1}.$ Thus the region $N\setminus
N_\pm$ is diffeomorphic to $T^2\times[0,1]$ and the contact structure
on it is easily seen to be a \dfn{basic slice}, see
\cite{Honda00a}. There are exactly two basic slices with given
dividing curves on their boundary and as there are two types of
stabilization of $L$ we see that the basic slice $N\setminus N_\pm$ is
determined by the type of stabilization done, and vice versa. Moreover
if $N$ is a standard neighborhood of $L$ then $L$ destabilizes if the
solid torus $N$ can be thickened to a solid torus $N_d$ with convex
boundary in standard form with dividing slope $\frac 1{\tb(L)+1}.$
Moreover the sign of the destabilization will be determined by the
basic slice $N_d\setminus N$. Finally, we notice that using
Theorem~\ref{bpaferry} we can destabilize $L$ by finding a bypass for
$N$ attached along a ruling curve whose slope is clockwise of
$1/(\tb(L)+1)$ (and anti-clockwise of $0$).

A neighborhood of a transverse knot $T$ can be modeled by the solid
torus $S_a=\{(\phi, (r,\theta))| r\leq a\}\subset S^1\times \R^2$ for
sufficiently small $a$, where $(r,\theta)$ are polar coordinates on
$\R^2$ and $\phi$ is the angular coordinate on $S^1$, with the contact
structure $\xi_{cyl}=\ker (d\phi + r^2\, d\theta)$.  Notice that the
tori $\partial S_b$ inside of $S_a$ have linear characteristic
foliations of slope $-b^2.$ Thus for all integers $n$ with $\frac 1{\sqrt
  n}<a$ we have tori $T_n=\partial S_{1/\sqrt{n}}$ with linear
characteristic foliation of slope $-\frac 1n.$ Let $L_n$ be a leaf of
the characteristic foliation of $T_n.$ Any Legendrian $L$ Legendrian
isotopic to one of the $L_n$ so constructed will be called a
\dfn{Legendrian approximation} of $T.$

\begin{lemma}[Etnyre-Honda 2001, \cite{EtnyreHonda01b}]\label{legapprox}
If $L_n$ is a Legendrian approximation of the transverse knot $T$ then
$(L_n)_+$ is transversely isotopic to $T.$ Moreover, $L_{n+1}$ is
Legendrian isotopic to the negative stabilization of $L_n.$\qed
\end{lemma}

This lemma is a key ingredient in the following result from which our
transverse classification results will follow from our Legendrian
classification results.

\begin{theorem}[Etnyre-Honda 2001, \cite{EtnyreHonda01b}]\label{tequivaletsl}
The classification of transverse knots up to transverse isotopy is
equivalent to the classification of Legendrian knots up to negative
stabilization and Legendrian isotopy.
\end{theorem}

\subsubsection{Classification strategy}\label{sssec:strategy}
The classification of Legendrian knots in a given knot type can be
done in a (roughly) three step process.

\smallskip

\noindent
{\bf Step I ---} {\em Identify the maximal Thurston-Bennequin
invariant of $\K$ and classify Legendrian knots realizing this.}

\noindent
{\bf Step II ---} {\em Identify and classify the non-maximal
Thurston-Bennequin Legendrian knots in $\K$ that do not destabilize
and prove that all other knots destabilize to one of these identified
knots.}

\noindent
{\bf Step III ---} {\em Determine which stabilizations of the maximal
Thurston-Bennequin invariant knots and non-destabilizable knots are
Legendrian isotopic.}

\smallskip 

As stabilization of a Legendrian knot is well defined and positive and
negative stabilizations commute, it is clear that these steps will
yield a classification of Legendrian knots in the knot type $\K$.

Step II is facilitated by the observation above that bypasses attached
to appropriate ruling curves of a standard neighborhood of a
Legendrian knot yield destabilizations. Similarly, if $L$ is a
Legendrian knot contained in a convex surface $\Sigma$ (and the
framing given to $L$ by $\Sigma$ is less than or equal to the framing
given by a Seifert surface) and there is a bypass for $L$ on $\Sigma$
then this leads to a destabilization of $L$. Moreover one can find
such a bypass in some cases by the Imbalance Principle discussed
above.

\subsubsection{Contact isotopy and contactomorphism}\label{sssec:isotopy}
We begin by recalling a result of Eliashberg concerning the
contactomorphism group of the standard contact structure $\xi_{std}$
on $S^3$. Fix a point $p$ in $S^3$ and let $\hbox{\em Diff}_0(S^3)$ be
the group of orientation-preserving diffeomorphisms of $S^3$ that fix
the plane $\xi_{std}(p),$ and let $\hbox{\em Diff}_{\xi_{std}}$ be the
group of diffeomorphisms of $S^3$ that preserve $\xi_{std}$.
\begin{theorem}[Eliashberg 1992, \cite{Eliashberg92a}]\label{thm:htpyequiv}
	The natural inclusion of 
	$$\hbox{Diff}_{\xi_{std}}\hookrightarrow \hbox{Diff}_0(S^3)$$ is a
	weak homotopy equivalence. \qed
\end{theorem} 

Using this fact it is clear that if one has a contactomorphism $\phi$
of $(S^3, \xi_{std})$ that takes a set $S\subset S^3$ to $S'\subset
S^3$, then there is a contact isotopy of $(S^3, \xi_{std})$ that takes
$S$ to $S'$. In particular, if one is trying to show that two
embeddings of a contact structure on a torus are contact isotopic then
one merely needs to construct a contactomorphism that takes one torus
to the other.  Similarly to show two Legendrian knots are Legendrian
isotopic one only needs to construct a contactomorphism that takes one
knot to the other (or takes a standard neighborhood of one of the
knots to the other, that is understand the contactomorphism type of
the complement of the standard neighborhood).

\subsection{Computations of $\tb, \rot$ and $\overline{\tb}$}\label{sec:cablefacts}
In this subsection we collect various facts that are useful in
computing the classical invariants of Legendrian knots on tori.

\subsubsection{Rotation numbers for curves on convex tori}\label{sec:rotationfunction}
Let $T$ be a convex torus in a contact manifold $(M,\xi),$ where $\xi$
has Euler class 0.  Now we define an invariant of homology classes of
curves on $T.$ Let $v$ be any globally non-zero section of $\xi$ and
$w$ a section of $\xi\vert_T$ that is transverse to and twists (with
$\xi$) along the Legendrian ruling curves and is tangent to the
Legendrian divides. If $\gamma$ is a closed oriented curve on $T$ then
set $f_T(\gamma)$ equal to the rotation of $v$ relative $w$ along
$\gamma.$ One may check the following properties (cf. \cite{Etnyre99,
EtnyreHonda01b}).
\begin{enumerate}
	\item The function $f_T$ is well-defined on homology classes.
	\item The function $f_T$ is linear.
	\item The function $f_T$ is unchanged if we isotope $T$
through convex tori in standard form.
	\item If $\gamma$ is a $(r,s)$-ruling curve or Legendrian
divide then $f_T(\gamma)=r(\gamma).$
\end{enumerate}

\subsubsection{Legendrian knots on tori}
We recall two simple lemmas from \cite{EtnyreHonda05}. The first
concerns the computation of the Thurston-Bennequin invariant for 
cables.

\begin{lemma} \label{tbcomp}   
Let $\K$ be a knot type and $N$ a solid torus representing $\K$ whose
boundary is a standard convex torus.  Suppose that $L\in
\mathcal{L}(\K_{(p,q)})$ is contained in $\partial N$.
\begin{enumerate}
\item Suppose $L_{(p,q)}$ is a Legendrian divide and 
$\mbox{slope}(\Gamma_{\partial N(\K)})={\frac qp}$.  Then 
\[
\tb(L_{(p,q)})=pq.
\] 
\item Suppose $L_{(p,q)}$ is a Legendrian ruling curve 
and $\mbox{slope} (\Gamma_{\partial N(\K)})={\frac{q'}{p'}}$. Then  
\[
\tb(L_{(p,q)})=pq- |pq'-qp'|.
\] 
\end{enumerate}\qed
\end{lemma}

A simple consequence of the discussion in
Subsection~\ref{sec:rotationfunction} yields the following computation
of the rotation number for cables.

\begin{lemma}  \label{rcomp}
Let $\K$ be a knot type and $N$ a solid torus representing $\K$ whose
boundary is a standard convex torus.  Suppose that $L\in
\mathcal{L}(\K_{(p,q)})$ is contained in $\partial N$.  Then $$
r(L_{(p,q)})= p \cdot r(\partial D) + q \cdot r(\partial \Sigma),$$
where $D$ is a convex meridional disk of $N$ with Legendrian boundary
on a contact-isotopic copy of the convex surface $\partial N$, and
$\Sigma$ is a convex Seifert surface with Legendrian boundary in
$\mathcal{L}( \K)$ which is contained in a contact-isotopic copy of
$\partial N(\K)$.  \qed
\end{lemma}

We end with a lemma that was established as Claim~4.2 in
\cite{EtnyreHonda05}. Recall that the {\em contact width} of a knot
type $\K$ is given by
\[
w(\K) = \textrm{sup}\frac{1}{\textrm{slope}(\Gamma_{\partial S})}
\]
where here $S$ ranges over all solid tori with convex boundary
representing $\K$.

\begin{lemma}\label{lem:maxtb}
Given a knot type $\K$, suppose $(r,s)$ is a pair of relatively prime
integers such that $\frac rs<w(\K)$. Then the maximal
Thurston-Bennequin invariant of $\K_{(r,s)}$ is
\[
\overline{\tb}(\K_{(r,s)})=rs.
\]\qed
\end{lemma}

\section{Solid tori in $S^3$}
\label{Classification of solid tori}
In Subsection~\ref{ss:nont} we classify non-thickenable tori in the
knot types of the positive torus knots, and in
Subsection~\ref{subsec:partiallythick} we classify the partially
thickenable tori. Subsection~\ref{knotsontori} discusses Legendrian
knots sitting on these tori as ruling curves and Legendrian dividing
curves.

\subsection{Non-thickenable tori}\label{ss:nont}

When considering tori $N$ that realize the knot type of $(p,q)$-torus
knot $\K,$ there are two different ``natural" coordinates to use. The
first is the longitude-meridian coordinates where the longitude comes
from the intersection of a Seifert surface with $\partial N.$ This
longitude will be called the $\infty$-longitude, and these coordinates
will be called the $\mathcal{C}$ coordinates. The other coordinate
system has the longitude given by the framing coming from the Heegaard
torus that $\K$ sits on in $S^3.$ This longitude will be called the
$\infty'$-longitude and these coordinates will be called the
$\mathcal{C}'$ coordinates. Except where stated otherwise we will
always use the more standard $\mathcal{C}$ coordinates.

\begin{lemma}
\label{basecasethickening}
Suppose that the solid torus $N$ represents the knot type of a
positive $(p,q)$-torus knot $\K$.  If $N$ has convex boundary then $N$
will thicken unless it has dividing slope
\[
e_k=\frac k {pq-p-q}
\]
for some $k \in \left\{1,2,\ldots\right\}$, and $2n_k$ dividing curves
where $n_k=\gcd (pq-p-q, k).$
\end{lemma}

\begin{proof} 
We begin by ignoring the contact structure and building a topological
model for the complement of $N.$ See Figure~\ref{fig:setup}. The knot
$\K$ can be thought to sit on a torus $T$ that separates $S^3$ into
two solid tori $V_1$ and $V_2$, each of which can be thought of as a
neighborhood of an unknot $F_1$ and $F_2.$ As $N$ is a neighborhood of
$\K$, we can isotope $T$ so that it intersects $N$ in an annulus and
thus $A'=T\setminus (T\cap N)$ is an annulus in the complement of $N$
with boundary on $\partial N.$ Moreover, there is a small neighborhood
of $A',$ which we denote ${N}(A')$ such that $S^3\setminus
(N\cup{N}(A'))$ consists of two solid tori, which we may think of as
$V_1$ and $V_2.$ Turning this construction around $V_1\cup
V_2\cup{N}(A')$ is the complement of $N.$ We can identify $N(A')$ as a
neighborhood of an annulus $A$ that has one boundary component a
$(p,q)$ curve on $\partial V_1$ and the other boundary component a
$(q,p)$ curve on $\partial V_2.$ Thus, topologically, the complement
of $N$ can be built as the neighborhood of two unknots (that form a
Hopf link) union the neighborhood of an annulus $A.$
\begin{figure}[ht]
        \relabelbox \small  {\centerline{\epsfbox{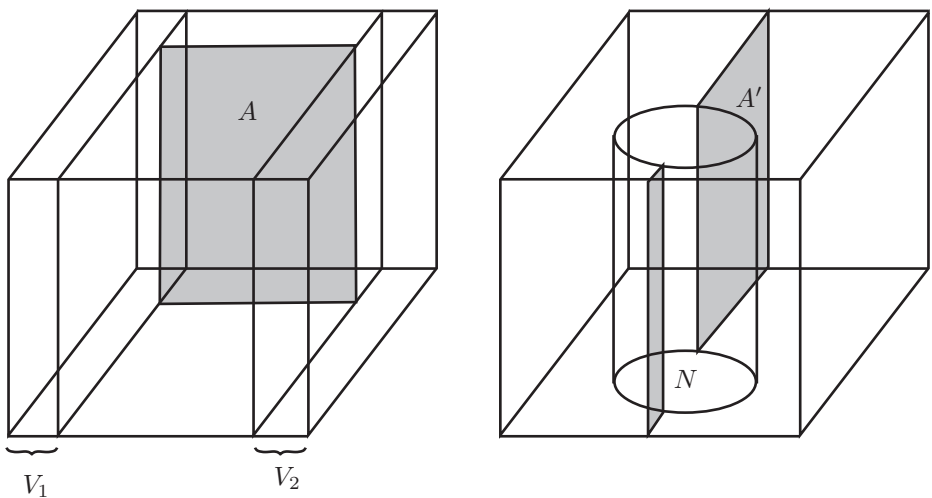}}}
  \relabel {1}{$V_1$}
  \relabel {2}{$V_2$}
  \relabel {a}{$A$}
  \relabel{p}{$A'$}
  \relabel{n}{$N$}
  \endrelabelbox
        \caption{The cubes in the picture represent $T^2\times[0,1]$
(the top and bottom are identified and the front and back are also
identified), thought of as the complement of the Hopf link $F_1\cup
F_2$. We have chosen coordinates on the torus so that the $(p,q)$
curve is vertical. On the left hand side we see the solid tori $V_1$
and $V_2$ (minus their cores) and the annulus $A$. On the
right hand side we see the solid torus $N$ and the annulus $A'$.}
        \label{fig:setup}
\end{figure}

Bringing the contact structure back into the picture we can assume
that $L_i$, $i=1,2$, is a Legendrian representative of $F_i$ in the
complement of $N,$ which maximize $tb(L_i)$ in the complement of $N$,
subject to the condition that $L_1 \sqcup L_2$ is isotopic to $F_1
\sqcup F_2$ in the complement of $N$. Let $tb(L_i)=-m_i$, where $m_i >
0.$ If $N(L_i)$ is a regular neighborhood of $L_i$, then
$\textrm{slope}(\Gamma_{\partial N(L_i)})=-1/m_i$ with respect to
$\mathcal{C}_{F_i}$.

Notice that $S^3\setminus (N(L_1)\cup N(L_2))$ is diffeomorphic to
$S=T^2\times [0,1]$ and contains $N.$ We wish to change coordinates on
$T^2$ so that $N$ is a vertical solid torus in $S.$ Specifically,
$T^2$ inherits coordinates as the boundary of $N(L_1),$ that is using
the coordinate system coming from the framing $\mathcal{C}_{F_1}.$ We
change coordinates so that the $(p,q)$ curve on $T^2$ becomes the
$(0,1)$ curve (which can be thought of as the longitude in the
$\mathcal{C}'$ framing). This can be done by sending the oriented
basis $((p,q),(p',q'))$ for $T^2,$ where $pq'-qp'=1$, to the basis
$((0,1),(-1,0)).$ This corresponds to the map
$\displaystyle \phi_1 = \left(\begin{array}{cc}
					q &-p\\
					q' &-p'\end{array}\right)$.
Then $\phi_1$ maps $(-m_1,1) \mapsto (-qm_1-p, -q'm_1-p')$.  Since we
are only interested in slopes, we write this as $(qm_1+p, q'm_1+p')$.

Similarly, we change from $\mathcal{C}_{F_2}$ to $\mathcal{C}'$.  The
only thing we need to know here is that $(-m_2,1)$ maps to
$(pm_2+q,p'm_2+q')$. Thus $S$ is a thickened torus $T^2\times[0,1]$
with dividing slope $\frac {q'm_1+p'}{qm_1+p}$ on $T\times\{0\}$ and
$\frac {p'm_2+q'}{pm_2+q}$ on $T\times \{1\}.$

Now suppose $qm_1+p \neq pm_2 + q$.  This would mean that the twisting
of Legendrian ruling representatives of $K$ on $\partial N(L_1)$ and
$\partial N(L_2)$ would be unequal.  Then we could apply the Imbalance
Principle to a convex annulus ${A}$ in $S^3 \backslash N$ between
$\partial N(L_1)$ and $\partial N(L_2)$ to find a bypass along one of
the $\partial N(L_i)$.  This bypass in turn gives rise to a thickening
of $N(L_i)$, allowing, by the twist number lemma \cite{Honda00a}, the
increase of $tb(L_i)$ by one.  Hence, eventually we arrive at $qm_1+p
= pm_2 + q$ and a standard convex annulus ${A}$; that is, the dividing
curves on $A$ run from one boundary component of $A$ to the other.

Since $m_i > 0$, the smallest solution to $qm_1+p = pm_2 + q$ is
$m_1=m_2=1$.  All the other positive integer solutions are therefore
obtained by taking $m_1=pj+1$ and $m_2=qj+1$ with $j$ a non-negative
integer.  We can then compute the boundary slope of the dividing
curves on $\partial(\widetilde{N})$ where $\widetilde{N}=N(L_1) \cup
N(L_2) \cup{N}(A).$ This will be the boundary slope for the solid
torus $\widetilde{N}$ containing $N$.  We have
\begin{equation}
\label{slopecalculation}
-\frac{q'(pj+1)+p'}{pqj+p+q} + \frac{p'(qj+1)+q'}{pqj+p+q}-\frac{1}{pqj+p+q} = -\frac{j+1}{pqj+p+q}
\end{equation}

After changing from $\mathcal{C}_\K'$ to $\mathcal{C}_\K$ coordinates,
and setting $k=j+1$, these slopes become $k/(pq-p-q)$ as desired.  We
also notice that $\partial \widetilde{N}$ has $2\gcd(pq-p-q, k)$
dividing curves. Thus any solid torus $N$ will thicken unless it
satisfies the conditions stated in the lemma.
\end{proof}

We have not yet proved that tori as described in the above lemma
actually exist. To rectify this problem we explicitly construct such
tori in the tight contact structure on $S^3$ by gluing together tight
contact structures on the pieces used in the proof of
Lemma~\ref{basecasethickening}. More specifically we have the
following.

\begin{construction}\label{constructNk}
let $N$ be a solid torus in the knot type of $\K$ and set
$M=S^3\setminus N.$ As noted in the proof above we can think of $M$ as
the union of two solid tori $V_1\cup V_2$ (which we think of as a
neighborhood of a Hopf link $F_1\cup F_2$), together with a product
neighborhood $N(A)$ of an annulus $A$ that has one boundary component
a $(p,q)$-curve on $\partial V_1$ and the other boundary component a
$(q,p)$-curve on $\partial V_2$. Also recall that $N(A)$ can be
thought of as a neighborhood of an annulus $A'$ that has boundary on
$N$ and that the union of $N$ and $N(A)$ is a thickened torus
$T^2\times[0,1]$ whose complement is $V_1\cup V_2.$

Now let $N_k^\pm$ denote $N$ with one of the two universally tight
contact structures on $N$ with convex boundary having 
boundary slope $s_k={\frac{k}{pq-p-q}}$
with respect to $\mathcal{C}$, and with $2n_k$ dividing curves. From the
classification of tight contact structures on solid tori this is
equivalent to the condition that the convex meridional disks all have
bypasses all of the same sign and thus the two contact structures on
$N_k^\pm$ differ by $-Id$.  (Note that when $k=1$ there is only one
contact structure. To avoid unnecessary notation we will frequently
write $N_1^\pm$ realizing that $N_1^+$ is the same as $N_1^-$.)

Let $N(A)=N(A')=A'\times [0,1]$ denote a product neighborhood of $A'$
and put a $[0,1]$-invariant contact structure on it, where the
dividing curves on $A' =A'\times\{\frac 12\}$ are in standard form.

The set $R=N_k\cup N(A')$ is diffeomorphic to $T^2\times[0,1]$ and we
can think of it as fibering over the annulus with fiber circles
representing the knot type $\K.$ For either choice of contact
structure on $N_k$, the contact structure on $R$ can be isotoped to be
transverse to the fibers of $R$, while preserving the dividing set on
$\partial R$.  It is well known, see for example \cite{Honda00b}, that
such a horizontal contact structure is universally tight.  Moreover,
we see the boundary conditions on $R$ are
$\#\Gamma_{T_1}=\#\Gamma_{T_2}=2$ and (with appropriately chosen
dividing curves on $A'$) $\mbox{slope}(\Gamma_{T_1})=-{\frac1{m_1}}$,
$\mbox{slope}(\Gamma_{T_2})=-m_2$ when using the coordinates on $T^2$
coming from the framing $\mathcal{C}_{F_1}$.

We know that there are exactly two universally tight contact
structures on $T^2\times[0,1]$ with these dividing curves, differing
by $-Id$, and their horizontal annuli contain bypasses all of the same
sign; one can easily see they correspond to the two choices of
universally tight contact structures on $N_k.$ We know that each of
these universally tight contact structures on $R$ embeds in the
standard tight contact structure as the region between a Legendrian
realization of the Hopf link $F_1\cup F_2.$ Thus the standard tight
contact structure on $S^3$ minus $R$ give standard neighborhoods of a
Legendrian realization $L_1$ of $F_1$, and $L_2$ of $F_2.$ Moreover,
we know that if $F_1$ and $F_2$ are oriented so that their linking is
$+1$ then for one choice of universally tight contact structure on $R$
we have that $L_1$ and $L_2$ are both obtained from maximal
Thurston-Bennequin unknots by only positive stabilizations and for the
other choice of universally tight contact structure on $R$ we have
only negative stabilizations. \hfill \qed
\end{construction}

We first notice that these $N_k^\pm$ just constructed in $S^3$ are
non-thickenable solid tori.

\begin{lemma}\label{lem:nonthicken}
The tori $N_k^\pm$ from Construction~\ref{constructNk} are
non-thickenable.
\end{lemma}
\begin{proof}
By Lemma~\ref{basecasethickening}, it suffices to show that $N_k$ does
not thicken to any $N_{k'}$ for $k' < k$. (We drop the $\pm$ from the
notation for $N_k$ for the remainder of this proof and just assume one
choice of sign is fixed throughout.) To this end, observe that the
$(p,q)$-torus knot is a fibered knot over $S^1$ with fiber a Seifert
surface $\Sigma$ of genus $g = (p-1)(q-1)/2$ (see \cite{Milnor68}).
Moreover, the monodromy map $\phi$ of the fibration is periodic with
period $pq$.  Thus, $M_k$ has a $pq$-fold cover $\widetilde{M}_k \cong
S^1 \times \Sigma$.  If one thinks of $M_k$ as $\Sigma \times [0,1]$
modulo the relation $(x,0) \sim (\phi(x),1)$, then one can view
$\widetilde{M}_k$ as $pq$ copies of $\Sigma \times [0,1]$ cyclically
identified via the same monodromy.  Now note that in $M_k$, the
$\infty'$-longitude intersects any given Seifert surface $pq$ times
efficiently.  It is therefore evident that we can view $M_k$ as a
Seifert fibered space with two singular fibers (the components of the
Hopf link).  The regular fibers are topological copies of the
$\infty'$-longitude, which itself is a Legendrian ruling curve on
$\partial M_k=\partial N_k$ with twisting $-(pq(k-1)+p+q)$.

We claim the pullback of the tight contact structure to
$\widetilde{M}_k$ admits an isotopy where the $S^1$ fibers are all
Legendrian and have twisting number $-(pq(k-1)+p+q)$ with respect to
the product framing.  To see this we consider the contact structure on
$V_i,$ the neighborhood of the Legendrian unknot $L_i$ (we will use
notation form Construction~\ref{constructNk}). In the $pq$-cover of
$M_k$ the torus $V_1$ will lift to $p$ copies of the $q$-fold cover
$\widetilde{V}_1$ of $V_1$ and similarly $V_2$ will lift to $q$ copies
of the $p$-fold cover $\widetilde{V}_2$ of $V_2.$ We can assume that
$\partial V_1$ has ruling slope $\frac qp$ (that is the ruling curves
are Legendrian isotopic to a Legendrian $\infty'$-curve on $\partial
M_k$) and similarly for $\partial V_2.$ The ruling curves lift to
curves of slope $\frac 1p$ in $\widetilde{V}_1.$ In particular they
are longitudes and have twisting $-(pq(k-1)+p+q).$ Moreover the
dividing curves on $\widetilde{V}_1$ are also longitudinal (a
different longitude of course). Thus we see that the contact structure
on $\widetilde{V}_1$ is just a standard neighborhood of one of the
ruling curves (pushed into the interior of the solid torus). Similarly
for $\widetilde{V}_2.$ Thus each of these tori is foliated by
Legendrian curves isotopic to the ruling curves. As $\widetilde{M}_k$
is made from copies of the $\widetilde{V}_i$ and copies of covers of
the convex neighborhoods of the annuli $A$ we see the claimed isotopy
of $\widetilde{M}_k$ so that the $S^1$ fibers are all Legendrian.

If $N_k$ can be thickened to $N_{k'}$, then there exists a Legendrian
curve topologically isotopic to the regular fiber of the Seifert
fibered space $M_k$ with twisting number greater than
$-(pq(k-1)+p+q)$, measured with respect to the Seifert fibration.
Pulling back to the $pq$-fold cover $\widetilde{M}_k$, we have a
Legendrian knot which is topologically isotopic to a fiber but has
twisting greater than $-(pq(k-1)+p+q)$.  Call this Legendrian knot
with greater twisting $\gamma$.  We will obtain a contradiction, thus
proving that $N_k$ cannot be thickened to $N_{k'}$.

Since $\Sigma$ is a punctured surface of genus $g$, we can cut
$\Sigma$ along $2g$ disjoint arcs $\alpha_i$, all with endpoints on
$\partial \Sigma$, that yield a polygon $P$.  Thus we have a solid
torus $S^1 \times P$ embedded in $\widetilde{M}_k$.  We first
calculate $\textrm{slope}(\Gamma_{\partial (S^1 \times P)})$ as
measured in the product framing.  To do so, note that a longitude for
this torus intersects $\Gamma,$ $2(pq(k-1)+p+q)$ times, and a meridian
for this torus is composed of $2$ copies each of the $2g$ arcs
$\alpha_i$, as well as $4g$ arcs $\beta_i$ from $\partial \Sigma$.
Now since $\partial \Sigma$ is a preferred longitude downstairs in
$M_k$, we know that $\Gamma$ intersects these $\beta_i$,
$2(pq-p-q)=2(2g-1)$ times positively.  But then the edge-rounding that
results at each intersection of an $S^1 \times \beta_i$ with an $S^1
\times \alpha_i$ yields $4g$ negative intersections with $\Gamma$.
Thus we obtain after edge-rounding that
$\textrm{slope}(\Gamma_{\partial (S^1 \times P)}) = -1/(pq(k-1)+p+q)$.

Now as in Lemma 3.2 in \cite{Honda00b}, we take $\widetilde{M}_k = S^1
\times \Sigma$ and pass to a (new) finite cover of the base by tiling
enough copies of $P$ together so that $\gamma$ is contained in a solid
torus $S^1 \times (\bigcup P)$.  We notice that $S^1 \times (\bigcup
P)$ is foliated by Legendrian knots with twisting $-(pq(k-1)+p+q)$
that are isotopic to the $S^1$ fibers in the product structure and
that the dividing curves on the boundary of the solid torus are
longitudinal. Thus $S^1 \times (\bigcup P)$ is a standard neighborhood
of a Legendrian curve with twisting $-(pq(k-1)+p+q)$ with respect to
the product structure. We know that inside any such solid torus any
Legendrian isotopic to the core of the torus has twisting less than or
equal to $-(pq(k-1)+p+q)$ (or else one could violate the Bennequin
bound). Thus $\gamma$ cannot exist.
\end{proof}

We now observe that the $N_k^\pm$ are the only candidates for
non-thickenable tori in the knot type of a positive $(p,q)$-torus
knot. In addition, we compute what the rotation numbers of Legendrian
curves on $\partial N_k^\pm$ are.

\begin{lemma}\label{candidates}\label{mustbeNk}
Let $N$ be a solid torus with convex boundary representing the
positive $(p,q)$-torus knot.  If $N$ does not thicken then $N$ must be
isotopic to one of the $N_k^\pm$ from Construction~\ref{constructNk}.

Moreover, if $\partial N_k^\pm$ is isotoped so that the ruling curves
are meridional then the meridional curves will have rotation number
$\pm (k-1)$, and if $\partial N_k^\pm$ is isotoped so that the ruling
curves are $\infty$-longitudes then the $\infty$-longitudes have
rotation number 0.
\end{lemma}
\begin{proof}
Let $N$ be a solid torus with convex boundary as in the lemma. If $N$ does not thicken then 
from the discussion in the proof of Lemma~\ref{basecasethickening} we
see that $S^3\setminus N$ can be thought of as the union of two solid
tori $V_1\cup V_2$ (which we think of as a standard neighborhood of a
Legendrian realization $L_1\cup L_2$ of the Hopf link $F_1\cup F_2$)
together with a product neighborhood $N(A)$ of an annulus $A$ that has
one boundary component a $(p,q)$-curve $K_1$ on $\partial V_1$ and the
other boundary component a $(q,p)$-curve $K_2$ on $\partial V_2$. From
the proof of Lemma~\ref{basecasethickening} we see that
$\tb(L_1)=-(p(k-1)+1)$ and $\tb(L_2)=-(q(k-1)+1)$ for some positive
integer $k.$ We can assume that $\partial A$ are ruling curves on the
tori $\partial V_1$ and $\partial V_2.$ Ruling curves on $A$ provide a
Legendrian isotopy form $K_1$ to $K_2.$ Thus $K_1$ and $K_2$ have the
same rotation numbers. From this and the discussion in
Construction~\ref{constructNk} we see that the signs of the
stabilizations must be the same, thus $r(L_1)=\pm p(k-1)$ and
$r(L_2)=\pm q(k-1).$ Hence $S^3\setminus N$ is contactomorphic to
$S^3\setminus N_k^\pm.$ Thinking of the neighborhood $N(A)$ as a
product neighborhood $N(A')$ of the annulus $A'$ (using the notation
from Lemma~\ref{basecasethickening} and
Construction~\ref{constructNk}) we see that $N\cup N(A')$ must be a
universally tight contact structure on $T^2\times[0,1]$ (or else we
could find a bypass for one of the $L_i$ and hence thicken $N$). We
will only get a universally tight contact structure on $N\cup N(A')$
if $N$ has convex meridian discs with bypasses all of the same sign,
as one may easily check by computing the relative Euler class of
$N\cup N(A')$.

The statement about meridional ruling curves is obvious. To verify the
statement for the $\infty$-longitudes we need to use the function
$f_T$ that measures the rotation numbers of curves on convex tori $T$
that was discussed in Subsection~\ref{sec:rotationfunction}. We fix
our attention on $N_k^+$ (leaving the analogous case for $N_k^-$ to
the reader). Recall $L_1$ is a Legendrian unknot obtained from the
maximal Thurston-Bennequin unknot by $p(k-1)$ positive
stabilizations. Thus if $V_1$ is a standard neighborhood of $L_1$ and
$K$ is a $(p,q)$-ruling curve on $\partial V_1$ then we see
\[
f_{\partial V_1}(K)=pf_{\partial V_1}(\mu')+qf_{\partial V_1}(\lambda'')=qp(k-1),
\]
where $\mu'$ is a meridional curve on $\partial V_1$ and $\lambda''$
is a longitude.

If we isotope $\partial N_k^+$ so that the ruling curves are
$\infty'$-curves then there is a convex annulus $A''$ in $S^3$ from
the curve $K$ on $\partial V_1$ to an $\infty'$-longitude $\lambda'$
on $\partial N_k^+$ that has dividing curves that run from one
boundary component to the other. Thus we can rule $A''$ by curves
parallel to $K$ and $\lambda'$ and see that $K$ and $\lambda'$ are
Legendrian isotopic. In particular $f_{\partial
N_k^+}(\lambda')=r(\lambda')=qp(k-1).$ Let $\lambda$ denote a
$\infty$-longitude on $\partial V_k^+.$ Since we know that
$\lambda=\lambda'-pq\mu$ where $\mu$ is a meridian on $\partial V_k^+$
we see that
\[
f_{\partial N_k^+}(\lambda)=f_{\partial N_k^+}(\lambda') - pq f_{\partial N_k^+}(\mu)=0.
\]  
\end{proof}

\begin{proof}[Proof of Theorem~\ref{thm:nonthick}]
The theorem merely collects the statements of
Lemmas~\ref{basecasethickening}, \ref{lem:nonthicken},
and~\ref{mustbeNk}, together with Construction~\ref{constructNk}.
\end{proof}

\subsection{Partially thickenable tori}
\label{subsec:partiallythick}
In this section we use the notation established in
Construction~\ref{constructNk} and the subsequent lemmas of the
previous section.  We notice that $M_k^\pm$ can always be constructed
so that it is contained in any arbitrarily small neighborhood of the
annulus $A$ union $N(L_1)\cup N(L_2)$ from
Construction~\ref{constructNk} and any two such constructed $M_k^\pm$
are isotopic (and hence the corresponding $N_k^\pm$ are isotopic too).

Throughout this subsection we will always be talking about tori in the
knot type of a positive $(p,q)$-torus knot.

\begin{lemma}\label{lem1}
Let $N$ be a solid torus in $N_k^\pm$ with standard convex boundary
having dividing slope $s\in [e_k,e_k^a).$ If $\gcd(k, pq-p-q)=1$, then
there can be no bypass $D$ inside $N_k^\pm\setminus N$ for $\partial
N_k^\pm$ attached along a ruling curve of slope $\infty'$.
\end{lemma}
\begin{proof}
Notice that $N_k^\pm\setminus N$ is diffeomorphic to $T^2\times
[0,1].$ Moreover the slope on $T^2\times\{0\}=\partial N$ is in
$[e_k,e_k^a)$ and on $T^2\times\{1\}=\partial N_k^\pm$ is $e_k.$ If
such a bypass existed then there would be a torus $T$ in $T^2\times
[0,1]$ with dividing slope $e_k^a$. Thus the contact structure on
$T^2\times[0,1]$ is not minimally twisting, but this is impossible as
the contact structure on $S^3$ we are considering is tight.

(Notice that if $\gcd(k, pq-p-q)>1$ then a bypass can be attached that
merely reduces the number of dividing curves.)
\end{proof}

\begin{lemma}\label{moveA}
Assume that $\gcd(k, pq-p-q)=1$. 
Let $L_1$ and $L_2$ be the two unknots used in the construction of
$M_k^\pm$ and $A$ the annulus, see Construction~\ref{constructNk}. Let
$N(L_1)\cup N(L_2)$ be the standard neighborhood of $L_1\cup L_2$ used
in this construction. Suppose that $\widehat A$ is any convex annulus
in the complement of $N(L_1)\cup N(L_2)\cup N$, which has boundary
Legendrian ruling curves parallel to $\partial A$ on $\partial N(L_1)
\cup \partial N(L_2)$, and such that $\widehat A$ is isotopic to $A$
in the complement of $N.$ Then the dividing curves on $\widehat A$ run
from one boundary component to the other and there is a contact
isotopy of $S^3$ taking $N(L_1)\cup N(L_2)\cup A$ to $N(L_1)\cup
N(L_2)\cup \widehat A.$
\end{lemma}
\begin{proof}
First notice that if $\widehat A$ is disjoint from $A$ then the first
statement is clear since if the dividing curves were not as stated
there would be a bypass for $N^\pm_k$ on a $\infty'$ ruling curve
contradicting Lemma~\ref{lem1}. (To see this recall that
$N_k^\pm=S^3-(N(L_1)\cup N(L_2)\cup N(A))$.) For the second statement
notice that there will be a diffeomorphism of $S^3$ fixing (set-wise)
$N(L_1)\cup N(L_2)$ and sending $A$ to $\widehat A.$ Moreover we can
assume this diffeomorphism preserves the dividing sets on $\partial
N(L_1)\cup N(L_2)$ and sends $\Gamma_A$ to $\Gamma_{\widehat A}.$ Thus
we may isotope the diffeomorphism so that it is a contactomorphism in
a neighborhood of $N(L_1)\cup N(L_2)\cup A.$ As the contact structure
on the complementary solid torus is unique (as indicated in the proof
of Lemma~\ref{mustbeNk}) we can further isotope this map to a
contactomorphism of $S^3$. As the space of contactomorphisms of the
standard contact structure on $S^3$ (that fix a point) is contractible
it is standard to find a contact isotopy as desired.

If $\widehat A$ and $A$ are not disjoint then we can use Isotopy
Discretization as discussed in Subsection~\ref{sssec:discritize} to
find a sequence of annuli $A_1,\ldots, A_n$ such that $A_1=A, A_n=A',$
each $A_i$ is a convex annulus with boundary Legendrian ruling curves
parallel to $\partial A$ and for each $i=1, \ldots, n-1, A_i$ and
$A_{i+1}$ are disjoint. The result now follows.
\end{proof}

\begin{proposition}\label{prop:pt}
Let $N$ be a solid torus in $N_k^\pm$ with standard convex boundary
having dividing slope $s\in [e_k,e_k^a).$ If $\gcd(k, pq-p-q)=1$, then
$N$ will thicken to a solid torus $N'$ of slope $e_k$ but not
beyond. Moreover, $N'$ is isotopic to $N_k^\pm.$
\end{proposition}

\begin{remark}\label{rem:pt}
Notice that if $\gcd(k,pq-p-q)>1$ then $N_k^\pm$ can be thinned to a
torus $N'$ that has the same dividing slope as $N_k^\pm$ but fewer
dividing curves. This will allow for the destabilization of the
Legendrian knots $L_1$ and $L_2$ used in
Lemma~\ref{basecasethickening}, which in turn, allow for the
thickening of $N'$ past $N_k^\pm$. Thus we see when $\gcd(k,
pq-p-q)>1$ that there are no partially thickenable tori in $N_k^\pm$.
\end{remark}

\begin{remark}
For the right handed trefoil knot there is another, arguably simpler,
proof of this result that is more in the spirit of the previous
subsection. We present a unified proof for all $(p,q)$ here and refer
to \cite{TosunThesis} for the alternate argument.
\end{remark}

\begin{proof}
Suppose that $N$ can be thickened past the slope $e_k$. Then it can be
thickened to $N_{k'}$ for some $k'<k.$ We can arrange $N$ to have
ruling curves isotopic to $\infty'$-longitudes. Taking an annulus
$\widetilde A$ from a ruling curve on $\partial N$ to a ruling curve
on $\partial N_{k'}$ (of slope $\infty'$) we see that there are enough
disjoint bypasses $D_1, \ldots, D_n$ along $\widetilde A$ for
$\partial N$ to thicken $N$ to a solid torus with dividing slope
outside the interval $[e_k,e_k^a).$ If the bypasses were contained in
$N_k$ this would of course be a contradiction, as 
we could attach them to $\partial N$ to obtain a convex torus in $N_{k}$
with slope $e_{k'}$. We now argue that we can isotope $N_k$ so that it
contains all the bypasses. This contradiction will imply that $N$
cannot be thickened to $N_{k'}$ for any $k'<k.$

To this end let $L_1$ and $L_2$ be the two unknots used in the
construction of $M_k^\pm$ and $A$ the annulus, see
Construction~\ref{constructNk}. From the construction we know that
$M_k^\pm$ is obtained by taking the union of arbitrarily small
neighborhoods $N(L_1)\cup N(L_2)$ of $L_1\cup L_2$ and $N(A)$ of $A$
(and rounding corners). Consider the 2-complex $X$ obtained from
$L_1\cup L_2$ by attaching (an extension of) $A.$ Clearly $M_k^\pm$
can be isotoped to be contained in any arbitrarily small neighborhood
of $X.$

We now consider the intersection of $X$ with the bypasses above. First
we notice there is a contact isotopy of the $D_i$ making them
transverse to $X.$ So the intersection consists of closed curves,
vertices (corresponding to the intersection of $D$ with $L_1\cup L_2$)
and arcs. We may now choose standard neighborhoods $N(L_i)$ of the
Legendrian knots $L_i$ (and possibly isotope the interiors of the
$D_i$) so that $N(L_i)$ intersects the bypass disks in disks (that is
each vertex of $X\cap D_i$ becomes a disk) that are disjoint from the
simple closed curves in $X\cap D_i$. We may now isotope $X$ so that
$X-(X\cap (N(L_1)\cup N(L_2)))$ is a convex annulus $\widehat A$ with
Legendrian boundary ruling curves on $\partial (N(L_1)\cup N(L_2))$
and intersects the bypass disks as $X-(X\cap (N(L_1)\cup N(L_2)))$
does.

Let $D$ denote one of these bypasses. We will show how to isotope $D$
to be disjoint from $M_k^\pm$ and observe that this argument can be
applied to each of the $D_i$ resulting in the desired
contradiction. It is clear that if $D\cap X=\emptyset$ then $D$ may be
assumed to be contained in $N_k^\pm.$ Thus we show how to eliminate
the intersections between $D$ and $X$. We first show how to remove the
closed curves from the intersection. Let $\gamma$ be an innermost
closed curve in $D\cap X$. (That is $\gamma$ bounds a disk on $D$ that
does not contain any other points of intersection between $X$ and
$D$.) Notice that from the set-up above $\gamma$ is an intersection
between $\widehat A$ and $D.$ We can isotope $\widehat A$, rel
boundary, so as to eliminate $\gamma$ from $X\cap D.$ (Notice along
the way, we might also eliminate some intersections between $X$ and
other $D_i$ but we do not increase the number of intersections between
$D_i$ and $X$.) By Lemma~\ref{moveA} we see that this isotopy can be
done by a contact isotopy, thus resulting in a new $X$ with all the
above properties but fewer intersections with the disk $D.$ Continuing
we can assume that $D\cap X$ contains no simple closed curves.

Now suppose that $\gamma$ is an arc in $D\cap X$ that connects two
vertices. We can take an interval in $\gamma$ that is disjoint from
the intersection of $D$ and $N(L_1)\cup N(L_2)$ and then isotope
$\widehat A$ as above to remove this interval from the intersection of
$X$ and $D.$ Thus $X\cap D$ consists of ``stars'' and arcs; that is,
each connected component of the intersection is either an arc (with
both endpoints on $\partial D$) or has a single vertex with several
edges (connecting the vertex to $\partial D$). We again notice that
the arcs of intersection are intersections between $\widehat A$ and
$D$ and thus we may remove them as above if they are outermost (that
is, separates off a disks from $D$ that does not contain any points of
intersection between $D$ and $X$).

We are now left to consider outermost ``stars''. Given such a star we
assume that the vertex comes from an intersection between $D$ and
$L_1.$ So we have a disk $D'\subset D\cap N(L_1)$ corresponding to the
vertex and the $p$ edges corresponding to $\widehat A\cap D$ (we would
have $q$ edges if $D$ intersected $L_2$ at the vertex under
consideration). Recall $M_k^\pm$ is obtained by taking the union of an
$I$ invariant neighborhood of $\widehat A$ and $N(L_1)\cup N(L_2)$ and
rounding corners. So we can isotope $D$ slightly near $\widehat A$ so
that $M_k^\pm \cap D$ consists of $D'$ union $p$ strips corresponding
to thickenings of the edges of $D\cap \widehat A$. From this it is
easy to see that $\partial (M_k^\pm\cap D)$ consists of $p$ arcs,
$\gamma_1,\ldots, \gamma_p.$ One of these arcs, which we denote
$\gamma_p$, divides $D$ into two disks, one of which contains all the
other $\gamma_i$'s (and no other intersections with $X$). Denote this
disk $C$. Notice that $N(L_2)$ does not intersect $C.$
\begin{figure}[ht]
        \relabelbox \small  {\centerline{\epsfbox{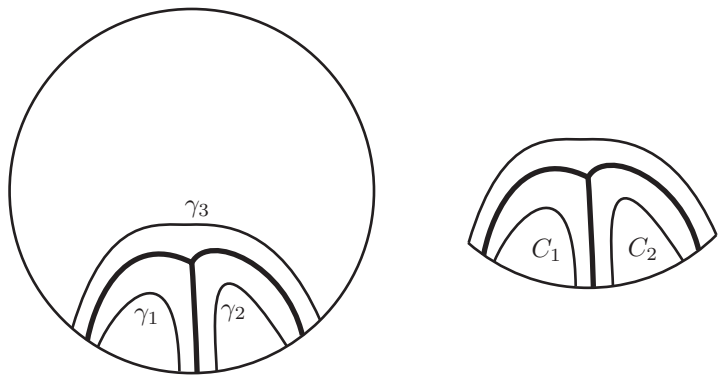}}}
  \relabel {1}{$\gamma_1$}
  \relabel {2}{$\gamma_2$}
  \relabel {3}{$\gamma_3$}
  \relabel{a}{$C_1$}
  \relabel{b}{$C_2$}
  \endrelabelbox
        \caption{On the left is the disk $D$ and a star component of
$D\cap X$ is in bold. The intersection of a neighborhood of $X$ with
$D$ is also shown along with the curves $\gamma_i$ that make up the
boundary of this region. On the right is the disk $C$ that $\gamma_3$
cuts off of $D$ and the sub disks $C_i$.}
        \label{fig:disk}
\end{figure}

Each arc $\gamma_i, i<p,$ separates a disk $C_i$ from $C$ that is
disjoint from the interior of $M_k^\pm.$ If we push $\partial M_k^\pm$
across the disk $C_i$ then we get a new torus $T'$ in $N_k^\pm-N.$
Recall that the ruling slope on $\partial N(L_2)$ was by
$(q,p)$-curves and that the isotopy of $\partial M_k^\pm$ to $T'$ can
be done fixing one of these curves. Thus the contact twisting of the
ruling curve is still $-(pq(k-1)+p+q)$, however, the ruling curve on a
convex torus with dividing slope in $[e_k, e_k^a)$ will always have
twisting less than or equal to $-(pq(k-1)+p+q)$ with equality if and
only if the dividing slope is $e_k.$ Thus we see that $T'$ has
dividing slope $e_k$ and hence is contact isotopic to $\partial
M_k^\pm.$ That is we can find a contact isotopy that eliminates one of
the arcs of intersection. Continuing in this way we push $\partial
M_k^\pm$ across the other disks $C_i$ by a contact isotopy resulting
in the disk $C$ being contained in $M_k^\pm.$ Now pushing $\partial
M_k^\pm$ across $C$ will not change the dividing set since $N_k^\pm$
is a non-thickenable torus. Combining these isotopies we have removed
the outermost ``star'' in $D\cap X.$

By successively removing outermost arcs or ``stars'' from $D\cap X$ we
can eventually make $D$ disjoint from $X$ and thus contained in
$N_k^\pm-N.$
\end{proof}

\begin{proposition}
Let $N$ be a solid torus in $N_k^\pm$ with standard convex boundary
having dividing slope $s\not \in [e_k,e_k^a).$ Then $N$ will thicken
to the solid torus $N_1$ (which is a standard neighborhood of the
maximal Thurston-Bennequin invariant Legendrian $(p,q)$-torus knot).
\end{proposition}

\begin{proof}
Given such a torus $N$ we know from the construction and discussion in
Subsection~\ref{ss:nont} that we can thicken $N$ to a solid torus $N'$
whose boundary is convex with two dividing curves of slope $e_k^a$ and
in the complement of $N'$ we will have $M_k^\pm$. Now taking an
annulus from $N'$ to $\partial N(L_1)$ (using the notation from
Construction~\ref{constructNk}) we will see that there is a bypass for
$\partial N(L_1)$ and thus we can increase the Thurston-Bennequin of
$L_1.$ As in the proof of Lemma~\ref{basecasethickening} we see that
$N'$ will thicken to some $N_{k'}^\pm$ with $k'<k.$ Thus we know we
can thicken past $N_{k'}^\pm$ unless $k'=1$, and hence we can thicken
to $N_1.$
\end{proof}

We are now ready to establish the main results stated in the
introduction concerning partially thickenable tori.
\begin{proof}[Proof of Theorem~\ref{thm:partially}]
The statements in the theorem just collect the facts from
Proposition~\ref{prop:pt}, Remark~\ref{rem:pt} and
Lemma~\ref{lem:basicrange}.
\end{proof}

\begin{proof}[Proof of Corollary~\ref{cor:tori}]
For statement (1) notice that if $n\leq s<n+1$ then a convex torus
with two dividing curves of slope $s$ will lie inside one of the
$N_m^\pm$ for $m=2,\ldots n$ or $N_1$. From the classification of the
$N_m^\pm$ we know there is a convex torus with two dividing curves and
infinite dividing slope inside each of the $N_m^\pm$ and it will
cobound with $\partial N_m^\pm$ a unique basic slice,
\cite{Honda00a}. Moreover there are two distinct such tori in $N_1$
and each of these two will cobound with $\partial N_1$ a unique basic
slice. Inside a basic slice there is a unique, up to contactomorphism,
convex torus of slope $s$. Thus given any convex torus $T$ with two
dividing curves of slope $s$ we can use this data to construct a
contactomorphism of $S^3$ taking $T$ to one of the tori described
above. Then the discussion in Subsection~\ref{sssec:isotopy} gives a
contact isotopy from $T$ to one of these tori. As there are $2n$ such
tori this establishes statement (1) of the theorem. 

The other statements in the corollary have analogous proofs. 
\end{proof}

\subsection{Legendrian knots on tori}\label{knotsontori}
In this section we prove two fundamental propositions about Legendrian
knots on partially thickenable, and non-thickenable, tori that will be
necessary in our classification of cables of torus knots.
\begin{proposition}\label{curvesonpt1}
Suppose $\K$ is a positive $(p,q)$-torus knot and $N_n^\pm$ is a solid
torus constructed above in Subsection~\ref{ss:nont}, for some $n>1$
with $\gcd{(n,pq-p-q)}=1$. Let $\frac sr\in[e_n,e_n^a)$ and $k= |\frac
sr \cdot e_n^a|$. If $T$ is the convex torus in $N_n^\pm$ with two
dividing curves and dividing slope $\frac sr$ and $L$ is a Legendrian
divide on $T$, then:
\begin{enumerate}
\item For any $y\in\N \cup\{0\}$ and $x< k$, any convex torus $T'$ on
which the Legendrian knot $S_\pm^xS_\mp^y(L)$ sits bounds a partially
thickenable, or non-thickenable, torus in $N_n^\pm$.
\item The Legendrian knot $S_\pm^kS_\mp^y(L)$ sits on a convex torus
$T'$ that bounds a solid torus that thickens to $N_1$.
\end{enumerate}
\end{proposition}
\begin{proof}
We will concentrate on the Legendrian divide $L$ on a torus $T$ inside
$N_n^+$ below, but analogous arguments also work for $N_n^-$. Recall
that inside the solid torus $N_n^+$ there is a convex torus $T'$ with
two dividing curves and dividing slope $e_n^a$. Let $L'$ be a
Legendrian ruling curve on $T'$ of slope $\frac sr$. Using an annulus
$A$ that $L$ and $L'$ cobound, it is easy to see that $L'$ is obtained
from $L$ by stabilizing $k=|\frac sr \cdot e_n^a|$ times.

We want to compute the difference between the rotation number of $L$
on $T$ and $L'$ on $T'$. The region between $T$ and $T'$ is a
thickened torus and the difference in these rotation numbers will be
given by the value of the relative Euler class of the thickened torus
evaluated on the annulus $A$. To compute this we use the
classification of tight contact structures on thickened tori, as given
in \cite{Honda00a}, and the fact that $N_n^+$ is universally tight. In
particular, we can compute the relative Euler class $e$ of the
thickened torus cobounded by $\partial N_n^+$ and $T$:
\[
P.D.(e)=( (r,s) -(b,a)) \in H_1(T^2\times I;\Z),
\]
where $P.D.$ stands for the Poincar\'e Dual and we are using the basis
for $H_1$ given by the meridian and longitude and $e_n^a=\frac ab$.
We can use this to compute the difference between the rotation number
of the $(r,s)$ curve on $\partial N_n^+$ and on $T$ which is
$(r(s-a)-s(r-b)=(sb-ra)=\frac sr \cdot \frac ab=\frac sr\cdot
e_n^a>0$. That is, $L'$ is obtained from $L$ by $k$ positive
stabilizations. According to Theorem~\ref{thm:partially} the solid
torus that $T'$ bounds can be thickened to $N_1$. As any further
negative stabilizations of $L$ can be seen on $T'$ as well (by having
$L$ intersect the dividing curves in a non-minimal way) we have
established the second point in the proposition.

For the first point in the proposition notice that the discussion
above shows that $S_+^xS_-^y(L)$, with $x<k$, cannot sit as a
Legendrian curve on a convex torus with dividing slope $e_n^a$ (since
otherwise $x\geq k$). Suppose that $S_+^xS_-^y(L)$ is also isotopic to
a curve on a convex torus $T'$ that is neither a partially
thickenable, nor a non-thickenable, torus in $N_n^+$. (This $T'$ is
not the same as in the previous paragraph.) We can extend the isotopy
of $S_+^xS_-^y(L)$ to an ambient contact isotopy and thus we may
assume that one fixed copy of $S_+^xS_-^y(L)$ sits on both a partially
(or non-) thickenable torus $T$ in $N_n^+$ and on a torus $T'$ that is
not a partially (or non-) thickenable torus in $N_n^+$. We may isotope
$T'$ near $S_+^xS_-^y(L)$ so that it agrees with $T$. Let $N$ be a
standard neighborhood of $S_+^xS_-^y(L)$ that intersects $T$ and $T'$
on a subset of $T\cap T'$. Let $A$ and $A'$ be the annuli in the
complement of $N$ given by $T$ and $T'$, respectively. We may further
assume that $\partial A=\partial A'$ are ruling curves on $\partial N$
and that all ruling curves on $\partial N$ are parallel to $\partial
A$. These annuli are properly topologically isotopic in the complement
of a neighborhood of $S_+^xS_-^y(K_+)$. (This follows from standard
results concerning incompressible surface in Seifert fibered spaces.)

We can use Isotopy Discretization as discussed in
Subsection~\ref{sssec:discritize} to find a sequence of annuli
$A_1,\ldots, A_m$ such that $A_1=A, A_m=A',$ each $A_i$ is a convex
annulus with boundary consisting of Legendrian ruling curves parallel
to $\partial A$ and for each $i=1, \ldots, m-1, A_i$ and $A_{i+1}$ are
disjoint and related by a bypass attachment. Notice that this gives us
a sequence of tori $T_1=T,\ldots, T_m=T'$ that are related by bypass
attachments in the complement of $S_+^xS_-^y(L)$. The torus $T_1$ is
partially (or non-) thickenable inside of $N_n^+$. We inductively show
that $T_i$ is also such a convex torus. Assume that we have shown that
$T_{i-1}$ is such a torus; then recall $T_i$ is obtained from
$T_{i-1}$ by attaching a bypass from the outside (that is from the
outside of the solid torus $T_{i-1}$ bounds) or from the inside. If we
attach the bypass to $T_{i-1}$ from the outside we get a new convex
torus that bounds a thickening of the solid torus that $T_{i-1}$
bounds, and so is also a partially (or non-) thickenable torus in
$N_n^+$. If we attach the bypass from the inside then as there is an
edge in the Farey tessellation between $e_n$ and $e_n^a$ (and the
dividing slope of $T_{i-1}$ is contained in the interval $[e_n,
e_n^a)$) we see that the dividing slope of $T_i$ is in $[e_n,e_n^a].$
But as in the previous paragraph the restriction on the rotation
number and Thurston-Bennequin invariant implies that the dividing
slope cannot be $e_n^a$. Thus the dividing slope of $T_i$ is in
$[e_n,e_n^a)$. In particular it bounds a partially (or non-)
thickenable solid torus in $N_n^+$. Thus $T_m=T'$ bounds a partially
(or non-) thickenable solid torus in $N_n^+$, which contradicts our
assumption on $T'$. From this we see that any convex solid torus on
which $S_+^xS_-^y(L)$ sits bounds a partially (or non-) thickenable
solid torus in $N_n^+$.
\end{proof}

\begin{proposition}\label{curvesonpt2}
Suppose $\K$ is a positive $(p,q)$-torus knot and $N_n^\pm$ is a solid
torus constructed above in Subsection~\ref{ss:nont}, for some $n>1$
with $\gcd{(n,pq-p-q)}=1$. Let $\frac sr\in(e^c_n,e_n)$ and $k= (\frac
sr \cdot e_n^a) - (\frac sr\cdot e_n)$. If $L$ is a ruling curve on
$\partial N_n^\pm$ with slope $\frac sr$, then:
\begin{enumerate}
\item For any $y\in\N \cup \{0\}$ and $x< k$ the convex torus
containing $S_\pm^xS_\mp^y(L)$ has dividing slope $e_n$ and is
contained in $N_n^\pm$.
\item The Legendrian knot $S^k_\pm S^y_\mp(L)$ sits on a convex torus
$T$ that bounds a solid torus that thickens to $N_1$.
\end{enumerate}
\end{proposition}
\begin{proof}
We will concentrate on a Legendrian ruling curve $L$ on $\partial
N_n^+$ below, but analogous arguments also work for $N_n^-$. The proof
of the second point in the proposition follows exactly as in the proof
of Proposition~\ref{curvesonpt1} and in particular, $S^k_+ S^y_-(L)$
sits on a convex torus $T'$ inside of $N_n^+$ with dividing slope
$e_n^a$. Moreover, any Legendrian knot that is a stabilization of $L$
that sits on $T'$ will have at least $k$ positive stabilizations.

The first point follows the same outline as the proof of Claim 6.5 in
\cite{EtnyreHonda05}, but is augmented by what we know from
Proposition~\ref{prop:pt}.  More specifically, if $T'$ also contains
$L$ and is isotopic to $\partial N_n^+$ then standard properties of
incompressible surfaces in Seifert fibered spaces (recall that the
sub-annulus of $T'$ contained in the complement of a neighborhood of
$L$ is incompressible in the complement of $L$) imply that $T'$ must
be isotopic to $\partial N_n^+$ relative to $L$.  Therefore, it
suffices to show that the slope of the dividing set does not change
under any isotopy of $\partial N_n^+$ relative to $L$.  Although we
would like to say that the isotopy leaves the dividing set of
$\partial N_n^+$ invariant, this is not true, see
\cite{EtnyreHonda05}, though we will show the dividing slope does not
change.  If $T'$ is isotopic to $\partial N_n^+$ relative to $L$ then
the standard Isotopy Discretization used above implies that there is a
sequence of surfaces $\Sigma_1=\partial N_n^+, \ldots \Sigma_m=T'$
such that each $\Sigma_i$ is convex and obtained from the previous
$\Sigma_{i-1}$ by a bypass attachment.  We inductively assume the
following:
\begin{enumerate}
\item $\Sigma$ is a convex torus which contains $L$ and satisfies
$2\leq \#\Gamma_\Sigma \leq 2(x+y)+2$ and $\mbox{slope}
(\Gamma_\Sigma) = e_n$.
\item $\Sigma$ is contained in a $[0,1]$-invariant $T^2\times[0,1]$
with $\mbox{slope} (\Gamma_{T_0}) = \mbox{slope} (\Gamma_{T_1}) = e_n$
and $\#\Gamma_{T_0} = \#\Gamma_{T_1}=2$ and is parallel to $T^2\times
\{i\}$.
\item There is a contact diffeomorphism
$\phi:S^3\stackrel\sim\rightarrow S^3$ which takes $T^2\times [0,1]$
to a standard $I$-invariant neighborhood of $\partial N_n^+$ and
matches up their complements.
\end{enumerate}
Notice that if we prove all the $\Sigma_i$ satisfy these conditions
then $T'$ will satisfy the conclusions of the first point of the
proposition, thus completing our proof.

We assume that $\Sigma_i$ satisfies the inductive hypothesis
above. Using the terminology from the proof of
Proposition~\ref{curvesonpt1} we notice that if a bypass is attached
to $\Sigma_i$ from the outside then the dividing slope cannot change
or this would give a thickening of our non-thickenable solid torus.
If the bypass is attached from the inside, then let $\Sigma'$ be the
torus obtained after the bypass is attached.  By Lemma~\ref{lem:eaec}
we see that $\mathfrak{s}=\mbox{slope}(\Gamma_{\Sigma'})$ must lie in
$[e_n, e_n^a]$. Since the argument in the first paragraph of this
proof disallows $\mathfrak{s}=e_n^a$, we know that $\mathfrak{s}\in
[e_n, e_n^a)$. Suppose that $\mathfrak{s} > e_n$. Let $\Sigma''$ be a
convex torus of slope $e_n^a$ and $\#\Gamma=2$ in the interior of the
solid torus bounded by $\Sigma'$.  Take a Legendrian curve $L'$ on
$\Sigma'$ which is parallel to and disjoint from $L$, and intersects
$\Gamma_{\Sigma'}$ minimally. (The existence of such a curve is easily
established by noting that $L$ is obtained from a curve $L'$ that
minimally intersects $\Gamma_{\Sigma'}$ by a sequence of ``finger
moves" across $\Gamma_{\Sigma'}$. Inducting on the number of such
moves one may show that a parallel copy of $L'$ can be made disjoint
from these moves.)  Similarly, consider $L''$ on $\Sigma''$.  Using
Lemma~\ref{lem:basicrange} we see that $|\Gamma_{\Sigma'}\cap L'| >
|\Gamma_{\Sigma''}\cap L''|$.  Thus an annulus that is bounded by $L'$
and $L''$ will contain bypasses for $\Sigma'$ that are disjoint from
$L$.  After successive attachments of such bypasses, we eventually
obtain $\Sigma'''$ of slope $e_n^a$ containing $L$, a contradiction.
Therefore (observing the restriction on the number of components of
$\Gamma_{\Sigma_i}$ are dictated by $\tb(S^k_+ S^y_-(L))$) we see that
Condition (1) is preserved.

Suppose $\Sigma'$ is obtained from $\Sigma$ by a single bypass move.
Since $\mbox{slope} (\Gamma_{\Sigma'}) = \mbox{slope}
(\Gamma_{\Sigma})$, either the bypass attachment was trivial or
$\#\Gamma$ is either increased or decreased by 2.  Suppose first that
$\Sigma'\subset N$, where $N$ is the solid torus bounded by
$\Sigma$. For convenience, suppose $\Sigma=T_{0.5}$ inside
$T^2\times[0,1]$ satisfies Conditions (2) and (3) of the inductive
hypothesis. In particular $T_1$ is a torus outside of $N$ with two
dividing curves. The tori $T_1$ and $\Sigma'$ cobound a thickened
torus $T^2\times[0.5,1]$ with non-rotative contact structure.  Thus by
the classification of tight contact structures on solid tori, we can
factor a non-rotative outer layer which is the new $T^2\times[0,0.5]$.
It is easy to see that this new $T^2\times [0,1]$ satisfies
Conditions~(2) and~(3) of the inductive hypothesis.

Now suppose $\Sigma'\subset (S^3\setminus N)$.  If $N'$ is the solid
torus $\Sigma'$ bounds then we prove that there exists a non-rotative
outer layer $T^2\times[0.5,1]$ for $S^3\setminus N'$, where
$\#\Gamma_{T_1}=2$.  This follows from repeating the procedure in the
proof of Lemma~\ref{basecasethickening}, where Legendrian
representatives of $F_1$ and $F_2$ were thickened and then connected
by a vertical annulus. This time the same procedure is carried out
with the provision that the representatives of $F_1$ and $F_2$ lie in
$S^3\setminus N'$.  Once the maximal thickness for representatives of
$F_1$ and $F_2$ is obtained, after rounding we get a convex torus in
$S^3\setminus N'$ parallel to $\Sigma'$ but with $\#\Gamma=2$.
Therefore we obtain a non-rotative outer layer $T^2\times[0.5,1]$.
\end{proof}

\section{Simple cables}
\label{simple cables}
In this section we classify the simple cables of positive torus
knots. These classification results and their proofs are very similar
to those in \cite{EtnyreHonda05} and the first two of them follow
directly from \cite{Tosun??}. We include sketches here to demonstrate
the classification strategy discussed in
Subsection~\ref{sssec:strategy} and as a warm-up for the more
intricate results in the next section.

\begin{theorem}\label{poscableclass}
Suppose $\K$ is a positive $(p,q)$-torus knot.  If $r,s$ are
relatively prime integers with
\[
\frac{r}{s}=\frac{1}{{s}/{r}}> w(\K)=pq-p-q,
\] 
then $\K_{(r,s)}$ is Legendrian simple.  Moreover, there is a unique
maximal Thurston-Bennequin invariant representative $L$ of
$\K_{(r,s)}$ which has invariants
\[
\tb(L)= \overline{\tb}(\K_{(r,s)})= rs-\left|w(\K)\cdot \frac{r}{s}\right| = rs- (r- s(pq-p-q)),
\] 
and $\rot(L)=0.$ All other Legendrian representatives of $\K_{(r,s)}$
destabilize to $L.$
\end{theorem}
\begin{proof}[Sketch of Proof]
We establish the theorem by (1) proving the above formula for
$\overline{\tb}(\K_{(r,s)}),$ (2) showing there is a unique Legendrian
knot $L$ with this as its Thurston-Bennequin invariant and (3) showing
that any other Legendrian knot in this knot type is a stabilization of
$L$.

To show (1) we let $K$ be any Legendrian knot in the knot type
$\K_{(r,s)}$. There is a solid torus $S$ realizing the knot type $\K$
that contains $K$ in $\partial S.$ We know there is a Seifert surface
for $\K_{(r,s)}$ with Euler characteristic $r+s(p+q-pq)-rs$ thus the
Bennequin inequality implies
\[
\tb(K)\leq rs-r+s(pq-p-q).
\]
From this we see that the twisting of the contact planes along $K$
measured with respect to $\partial S$ is less than or equal to
$-r+s(pq-p-q).$ Our condition that $r/s> pq-p-q$ implies that
$-r+s(pq-p-q)<0$, from which we can conclude that $\partial S$ can be
made convex without moving $K.$ Let $a$ be the slope of the dividing
curves on $\partial S.$ We know $a\geq w(\K)$ or negative.  Moreover,
$|a\cdot \frac rs| \geq |w(\K)\cdot \frac rs|$ with equality if and
only if $a=w(\K)$. Since we know that $\tb(K)$ is $rs$ plus
$tw(K,\partial S)$ and $tw(K,\partial S)$ is $-|a\cdot \frac rs|$
times the number of dividing curves, we clearly see that the maximal
possible Thurston-Bennequin invariant is realized on the the boundary
of a solid torus $S$ with convex boundary having two dividing curves
of slope $\frac 1{w(\K)}.$ If $S$ is the standard neighborhood of a
Legendrian knot in the knot type $\K$ with maximal Thurston-Bennequin
invariant then a ruling curve of slope $\frac sr$ will give a
Legendrian knot $L$ in the knot type $\K_{(r,s)}$ realizing this bound
as its Thurston-Bennequin invariant. Thus we have computed
$\overline{\tb}(\K_{(r,s)})$. Notice we have also shown that if $K$ is
any other Legendrian knot with $\tb(K)=\overline{\tb}(\K_{(r,s)})$
then $K$ will sit on the boundary of a standard neighborhood of a
maximal Thurston-Bennequin invariant Legendrian knot representing
$\K$. Since there is a unique such knot, standard arguments, like
those in \cite{EtnyreHonda01b, EtnyreHonda05} and discussed in
Subsection~\ref{sssec:isotopy}, show that $K$ is Legendrian isotopic
to $L.$ Thus we have shown there is a unique Legendrian representative
with maximal Thurston-Bennequin invariant.

We are left to check (3). To this end let $K$ be a Legendrian knot in
the knot type $\K_{(r,s)}$ with $\tb(K)< \overline{\tb}(\K_{(r,s)})$
and let $S$ be a solid torus in the knot type $\K$ such that $K$ sits
on $\partial S.$ As mentioned above we can assume that $\partial S$ is
convex. Let $a$ be the dividing slope for $\partial S$. If $a$ is
positive then there is some integer $n\geq 0$ such that $\frac 1{n+1}<
a\leq\frac 1n$. (A similar argument will hold for $a$ negative.) Thus
there is a convex torus $T$ inside $S$ with two dividing curves of
slope $\frac 1n$. As $\frac sr\cdot \frac 1n\leq \frac sr \cdot b$ for
any slope $b\in (\frac 1{n+1}, \frac 1n]$ with equality if and only if
$b=\frac 1n,$ we see that the $(r,s)$ ruling curve on $T$ has
Thurston-Bennequin invariant less than or equal to $\tb(K)$ and it is
strictly less than $\tb(K)$ unless $a=\frac 1n$. Taking an annulus
between $K$ and a ruling curve on $T$ we can find a bypass to show
that $K$ destabilizes unless $a=\frac 1n$. In this case we can assume
that $T$ is $\partial S$ and $S$ is a standard neighborhood of a
Legendrian knot in the knot type $\K.$ As $\K$ is Legendrian simple
and $n$ is not the maximal Thurston-Bennequin invariant we can thicken
$S$ to a solid torus $S'$ that is a standard neighborhood of a
Legendrian knot with $\tb=n+1.$ We can now use the ruling curve on
$\partial S'$ to show that $K$ destabilizes.
\end{proof}

\begin{theorem}\label{negcableclass}
Suppose $\K$ is a positive $(p,q)$-torus knot.  If $r,s$ are
relatively prime integers with $s>1$ and $\frac sr<0$, then
$\K_{(r,s)}$ is also Legendrian simple.  Moreover,
$\overline{tb}(\K_{(r,s)})=rs$ and the set of rotation numbers
realized by
$\{L\in \L(\K_{(r,s)})| tb(L)=\overline{tb}(\K_{(r,s)})\}$ is 
\[\{\pm (r + s(n+k)) \mbox{ } | \mbox{ } k=(pq-p-q+n), (pq-p-q+n)-2, \ldots, -(pq-p-q+n) \},\]
where $n$ is the integer that satisfies
$$-n-1 < {\frac rs}< -n.$$
All other Legendrian knots destabilize to one of these maximal
Thurston-Bennequin knots.
\end{theorem}
Notice that the restriction $s>1$ is reasonable as when $s=1$ we know
$\K_{(r,s)}=\K.$
\begin{proof}[Sketch of Proof]
This theorem is essentially Theorem~3.6 from \cite{EtnyreHonda05}, the
only difference being that $\K$ is not uniformly thick. As we saw in
the previous proof the only real difference in this case where $\K$ is
not uniformly thick is that we have to be careful to argue that
Legendrian knots with non-maximal Thurston-Bennequin invariants
destabilize. But in this case we see that if $K$ is any Legendrian
knot in the knot type $\K_{(p,q)}$ then it sits on a convex torus $T$
bounding a solid torus $S$ in the knot type $\K$ and there is either a
torus $T'$ parallel to $T$ inside $S$ or outside $S$ such that $T'$ is
convex with dividing slope $\frac sr.$ We can use $T'$ to find a
destabilization of $K.$
\end{proof}

\begin{theorem}\label{intermediatecableclass}
  Suppose $\K$ is a positive $(p,q)$-torus knot with
  $(p,q)\not=(2,3)$.  If $r,s$ are relatively prime positive integers
  with $0 <\frac rs<w(\K)=pq-p-q$ but $\frac sr\not\in J$, where $J$
  is as in Theorem~\ref{main}, then $\K_{(r,s)}$ is also Legendrian
  simple.  Moreover, $\overline{tb}(\K_{(r,s)})=rs$ and the set of
  rotation numbers realized by $\{L\in \L(\K_{(r,s)})|
  tb(L)=\overline{tb}(\K_{(r,s)})\}$ is
\[\{\pm (r + s(-n+k)) \mbox{ } | \mbox{ } k=(pq-p-q-n), (pq-p-q-n)-2, \ldots, -(pq-p-q-n) \},\]
where $n$ is the integer that satisfies
$$n-1 < {\frac rs}< n.$$
All other Legendrian knots destabilize to one of these maximal
Thurston-Bennequin knots. 
\end{theorem}
\begin{proof}[Sketch of Proof]
Establishing the classification of maximal Thurston-Bennequin
Legendrian knots in this knot type can be done exactly as in
Theorem~3.6 from \cite{EtnyreHonda05}, see \cite{Tosun??} for details,
except when $\frac sr \in [e_n, e_n^a)$ for some $n$ not relatively
prime to $pq-p-q$. If $L$ is a Legendrian knot in the knot type
$\K_{(r,s)}$ for such an $\frac sr\not= e_n$ and $L$ has maximal
Thurston-Bennequin invariant, then, as discussed above, $L$ will sit
as a Legendrian divide on a convex torus $T$ in the knot type
$\K$. Such a torus bounds a solid torus $S$ that can be thickened to a
solid torus with convex boundary having two dividing curves of slope
$e_n$. As mentioned in Corollary~\ref{cor:tori}, see also
Remark~\ref{rem:pt}, we see that this torus further thickens to
$N_1$. Thus the reasoning in Theorem~3.6 in \cite{EtnyreHonda05}
applies. If $L$ is a Legendrian knot in the knot type $\K_{(r,s)}$
with $\frac sr=e_n$, then it again sits as a Legendrian divide on a
convex torus $T$. If $T$ is not $\partial N_n^\pm$ then according to
Corollary~\ref{cor:tori} it will bound a solid torus that thickens to
$N_1$. If $T=\partial N_n^\pm$ then since $e_n\not\in J$, by
assumption, we know $\gcd(n, pq-p-q)\not=1$ and hence $T$ has more
than two dividing curves. Below we show that we can find a torus $T'$,
inside the solid torus $T$ bounds, with two less dividing curves on
which $L$ also sits. Of course this new torus will thicken to $N_1$
and hence we are done as above. To find $T'$ notice that according to
the classification of contact structures on thickened tori we can find
a convex torus $T_0$ inside of $S$, the solid torus $T$ bounds, with
two dividing curves of slope $e_n$. Let $B=T_0\times [0,1]$ be the
thickened torus that $T$ and $T_0$ cobound. Take a simple closed curve
$\gamma$ on $T_0$ that intersects a curve of slope $e_n$ one time. Let
$A=\gamma\times [0,1]$ be an annulus in $B$ running from $\gamma$ on
$T_0$ to $T$. We can arrange that $\partial A$ consists of ruling
curves on $T_0$ and $T$. Now if $\gcd(n,pq-p-q)>2$ then there will be
at least 2 non-adjacent bypasses on $A$ for $T$. Thus one of them will
be disjoint from $L$. Pushing $T$ across this bypass will result in
the torus $T'$ with fewer dividing curves than $T$ and on which $L$
sits. Since we are considering $(p,q)$-torus knots notice that
$pq-p-q$ is odd and thus $\gcd(n,pq-p-q)$ cannot be even, thus the
condition that $\gcd(n,pq-p-q)>2$ is satisfied.

We are left to show that any Legendrian knot with non-maximal
Thurston-Bennequin invariant destabilizes. Let $K$ be a Legendrian
knot in the knot type $\K_{(r,s)}$ with $\tb(K)< rs.$ We know that $K$
can be put on a convex torus $T$ that bounds a solid torus $S$
representing the knot type $\K.$ Let $a$ be the dividing slope of $T.$
If $a>\frac sr$ then there is a torus $T'$ parallel to $T$ inside $S$
with dividing slope $\frac sr.$ We can use an annulus that cobounds
$K$ and a Legendrian divide on $T'$ to show that $K$ destabilizes. Now
suppose that $a<\frac sr.$ If $a\in I_n = [e_n,e_n^a)$ for some $n$
then from Lemma~\ref{lem:basicrange} we see that $|a\cdot \frac
sr|\geq |e_n^a\cdot \frac sr|$ with equality if and only if
$a=e_n^a$. Since $a\not=e_n^a$ we can let $T'$ be a torus inside $S$
that is parallel to $T$ and has dividing slope $e_n^a$ and use an
annulus between $K$ and a ruling curve on $T'$ to show $K$
destabilizes. If $a$ is not in $I_n=[e_n,e_n^a)$ for any $n$ then from
Theorem~\ref{thm:partially} we know there is a torus $T'$ outside $S$
that is parallel to $T$ and has dividing slope $\frac 1{pq-p-q}$. Thus
between $T$ and $T'$ we have a convex torus $T''$ with dividing slope
$\frac sr.$ As above we can use this torus to show $K$ destabilizes.
\end{proof}

\section{Cables of positive torus knots (other than the trefoil)}
\label{non-simple cables non-trefoil}
Recall if $\K$ is the knot type of the positive $(p,q)$-torus knot and
$(p,q)\not=(2,3)$ then we set
\[
e_k=\frac k{pq-p-q},
\]
$J_k=(e_k^c, e_k^a)$, $\mathcal{I}=\{n\in \Z: n>1 \text{ and }
\gcd(n,pq-p-q)=1 \}$ and $J=\cup_{n\in \mathcal{I}} J_n.$ Much of
Theorem~\ref{main} was proven in the previous section. To complete the
proof we need to classify Legendrian knots in the $(r,s)$-cable of the
$(p,q)$-torus knot type $\K$ when $\frac sr\in J_n$ for some $n\in
\mathcal{I}.$ In the next two propositions we do this first for the
case when $\frac sr \in [e_n,e_n^a)$, and then for the case when
$\frac sr \in (e_n^c,e_n)$.
\begin{proposition}\label{prop:classify1}
With the notation above, suppose $\frac sr\in [e_n, e_n^a)$ for some
$n\in \mathcal{I}.$ Then there is some $k\geq 0$ such that $\frac
1{k-1}> \frac sr >\frac 1{k}$ and $\mathcal{L}(\K_{(r,s)})$ admits the
following classification.
\begin{enumerate}
\item The maximal Thurston-Bennequin invariant is
$\overline{\tb}(\K_{(r,s)})=rs$.
\item For each integer $i$ in the set 
\[\{\pm (r + s(-k+l)) \mbox{ } | \mbox{ } l=(pq-p-q-k), (pq-p-q-k)-2, \ldots, -(pq-p-q-k) \},\]
there is a Legendrian $L_i\in \mathcal{L}(\K_{(r,s)})$ with 
\[
\tb(L_i)=rs\,\, \text{ and } \rot(L_i)=i.
\]
\item There are two Legendrian knots $K_\pm\in
\mathcal{L}(\K_{(r,s)})$ with
\[
\tb(K_\pm)=rs \,\, \text{ and }  \rot(K_\pm)=\pm(s(pq-p-q)-r).
\]
\item All Legendrian knots in $\mathcal{L}(\K_{(r,s)})$ destabilize to
one of the $L_i$ or $K_\pm.$
\item\label{item5} Let $c=(\frac sr \cdot e_n^a)-1$.  For any $y\in\N
\cup \{0\}$ and $x\leq c$ the Legendrian $S_\pm^xS_\mp^y(K_\pm)$ is
not isotopic to a stabilization of any of the other maximum
Thurston-Bennequin invariant Legendrian knots in
$\mathcal{L}(\K_{(r,s)})$.
\item Any two stabilizations of maximal Thurston-Bennequin invariant
Legendrian knots in $\mathcal{L}(\K_{(r,s)})$, except those mentioned
in item~(\ref{item5}), are Legendrian isotopic if they have the same
$\tb$ and $\rot$.
\end{enumerate}
\end{proposition}
\begin{proof} We follow the standard approach to classifying
Legendrian knots in a given knot type outlined in
Section~\ref{sec:pre}.

\medskip

\noindent
{\bf Step I ---} {\em Identify the maximal Thurston-Bennequin
invariant of the knot type and classify Legendrian knots realizing
this:} The computation of the maximal Thurston-Bennequin invariant is
done in Lemma~\ref{lem:maxtb}.
 
\smallskip

\noindent
{\em $\bullet$ Construction of maximal Thurston-Bennequin invariant
knots in $\mathcal{L}(\K_{(r,s)})$:} Let $N_m^\pm$ be the
non-thickenable solid tori representing $\K$ that were constructed in
Subsection~\ref{ss:nont}. Recall $N_1$ is a standard neighborhood of
the maximal Thurston-Bennequin invariant Legendrian $(p,q)$-torus knot
$L$ (and that there is only one $N_1$ so the $\pm$ is ignored
here). Inside $N_1$ there are solid tori corresponding to stabilizing
$L$, $(pq-q-p)-k$ times. The range of the rotation numbers for the
Legendrian $(p,q)$-torus knots represented by these tori is
$S=\{(pq-p-q-k), (pq-p-q-k)-2, \ldots, -(pq-p-q-k)\}$. Denote these
tori $S_l$ for $l\in S.$ Inside each $S_l$ there are two tori
$S_l^\pm$ that come from positively or negatively stabilizing the
Legendrian knot corresponding to $S_l$. In the thickened torus
$S_l-S_l^\pm$ there is a unique convex torus $T_l^\pm$ with dividing
slope $\frac sr$. Let $i=sl \pm m$ where $m=r-sk>0$ is the
remainder. Denote by $L_i$ a Legendrian divide on $T_l^\pm$. We
clearly have that $\tb(L_i)=rs$ and the computation in the proof of
Lemma~3.8 in \cite{EtnyreHonda05} (or similar to the one given below
for $K_\pm$) gives that $\rot(L_i)=i$.

Now consider the two tori $N_n^\pm$. Inside each one there is a convex
torus $T^\pm$ with dividing slope $\frac sr.$ Let $K_\pm$ be a
Legendrian divide on $T^\pm$. Again it is clear that $\tb(K_\pm)=rs.$
Recall that from Lemma~\ref{rcomp} we know that
\[
\rot(K_\pm)=r\rot(\partial D) + s\rot(\partial \Sigma)
\]
where $D$ is a meridional disk for $T^\pm$ with Legendrian boundary
and $\Sigma$ is a surface, outside the solid torus $T^\pm$ that
bounds, with Legendrian boundary on $T^\pm.$ If $D'$ and $\Sigma'$ are
the corresponding surfaces for $\partial N_n^\pm$ then we know from
Lemma~\ref{candidates} that $\rot (\partial D')=\pm(n-1)$ and
$\rot(\partial \Sigma')=0.$ Thus the rotation number of an
$(r,s)$-ruling curve on $\partial N_n^\pm$ is $\pm r(n-1)$. To compute
the rotation number for the Legendrian divide on $T^\pm$ we use the
classification of tight contact structures on thickened tori, as given
in \cite{Honda00a}, and the fact that $N_n^\pm$ is universally
tight. In particular, we can compute the relative Euler class $e$ of
the thickened torus cobounded by $N_n^\pm$ and $T^\pm$:
\[
P.D.(e)=\pm( (r,s) -( pq-p-q, n)) \in H_1(T^2\times I;\Z),
\]
where $P.D.$ stands for the Poincar\'e Dual and we are using the basis
for $H_1$ given by the meridian and longitude.  We can use this to
compute the difference between the rotation number of the $(r,s)$
curve on $\partial N_n^\pm$ and on $T^\pm$ which is
$\pm(r(s-n)-s(r-(pq-p-q))$. Thus we have that $\rot(K_\pm)=\pm
(s(pq-p-q)-r)$.

\smallskip

\noindent
{\em $\bullet$ Classification of maximal Thurston-Bennequin invariant
knots in $\mathcal{L}(\K_{(r,s)})$:} If $K\in \mathcal{L}(K_{(r,s)})$
with $\tb(K)=rs$ then $K$ sits on a convex torus with dividing slope
$\frac sr$. Theorem~\ref{thm:partially} and Corollary~\ref{cor:tori}
say that such a torus is one of the ones considered when constructing
$K_\pm$ and $L_i$. Thus, a by now standard argument, see
\cite{EtnyreHonda01b} and Subsection~\ref{sssec:isotopy} above, says
the torus must be isotopic to one of the ones used in those
constructions from which we can also conclude that $K$ is isotopic to
one of $K_\pm$ or $L_i$.

\medskip

\noindent
{\bf Step II ---} {\em Prove all non-maximal Thurston-Bennequin
invariant knots in $\mathcal{L}(\K_{(r,s)})$ destabilize:} Let $K$ be
any Legendrian knot in $\mathcal{L}(\K_{(r,s)})$ with
Thurston-Bennequin invariant less than $rs$. Let $T$ be a torus
bounding a solid torus $S$ in the knot type $\K$ on which $K$
sits. Since $\tb< rs$ we know that we can perturb $T$ relative to $K$
so that it is convex. If the dividing slope $t$ of $T$ is equal to
$\frac sr$ then $K$ intersects the dividing curves inefficiently and
we can find a bypass for $K$ on $T$. Thus we can destabilize $K$. If
$t\not=\frac sr$ then we have three cases to consider. Case one is
when $t\not\in [e_m, e_m^a)$ for any $m$. In this case
Theorem~\ref{thm:partially} tells us that $S$ can be thickened to a
standard neighborhood of a maximal Thurston-Bennequin knot in
$\mathcal{L}(\K)$. Thus there is a convex torus $T'$ parallel to $T$
(either inside $S$ or outside $S$ depending on $t$) with dividing
slope $\frac sr.$ We can use an annulus between $T$ and $T'$ with
boundary on $K$ and a Legendrian divide on $T'$ to find a bypass for
$K $ and hence $K$ destabilizes. Case two is when $t\in [e_m,e_m^a)$
for $m\not=n.$ Lemma~\ref{lem:basicrange} says that $|t\cdot \frac
sr|$ is strictly greater than $|\frac sr \cdot e_m^a|$ and $|\frac sr
\cdot e_m^c|$ (since $t$ is on the interior of $[e_m^c, e_m^a]$). Thus
there is a torus $T'$ in $S$ with dividing slope $e_m^a$. Using an
annulus between $K$ on $T$ and a $\frac sr$ ruling curve on $T'$ we
find a bypass for $K$ and hence a destabilization. Finally in case
three we consider $t\in [e_n, e_n^a)$. In this case we can find a
torus $T'$ as in case one to destabilize $K$.
 
\medskip

\noindent
{\bf Step III ---} {\em Determine which stabilizations of the $K_\pm$
and $L_i$ are Legendrian isotopic:} We first notice that exactly as in
Lemma 4.12 of \cite{EtnyreHonda01b} and Theorem~3.6 in
\cite{EtnyreHonda05} we see that stabilizations of the $L_i$ are
Legendrian isotopic whenever they have the same Thurston-Bennequin
invariants and rotation numbers. (Recall this is easily established by
showing that when two of the $L_i$ are stabilized a minimal number of
times to have the same invariants they can both be realized as a
ruling curve on the boundary of a standard neighborhood of the same
Legendrian knot in the knot type $\K$.)

We will now concentrate on $K_+$ below, but analogous arguments also
work for $K_-$. From Proposition~\ref{curvesonpt1} we see that
$S_+^xS_-^y(K_+)$ sits on a torus $T$ that bounds a solid torus that
thickens to $N_1$ if $x>c$. In particular $T$ sits inside a solid
torus $S$ used in the construction of one of the $L_i$ (that is $L_i$
is a Legendrian dividing curve on $\partial S$). Thus we may use an
annulus that $S_+^xS_-^y(K_+)$ and $L_i$ cobound to see that
$S_+^xS_-^y(K_+)$ destabilizes to $L_i$.

We are left to see that $S_+^xS_-^y(K_+)$ is not isotopic to any
stabilization of the other maximal Thurston-Bennequin invariant knots
if $x\leq c.$ But this is clear from part one of
Proposition~\ref{curvesonpt1} since any stabilization of one of the
$L_i$ or $K_-$ sits on a convex torus that does not bound a partially
(or non-) thickenable torus contained in $N_n^+$.
\end{proof}

\begin{proposition}\label{prop:classify2}
With the notation above, suppose $\frac sr\in (e_n^c, e_n)$ for some
$n\in \mathcal{I}.$ Then there is some $k\geq 0$ such that $\frac
1{k-1}> \frac sr >\frac 1{k}$ and $\mathcal{L}(\K_{(r,s)})$ admits the
following classification.
\begin{enumerate}
\item The maximal Thurston-Bennequin invariant is
$\overline{\tb}(\K_{(r,s)})=rs$.
\item For each integer $i$ in the set 
\[\{\pm (r + s(-k+l)) \mbox{ } | \mbox{ } l=(pq-p-q-k), (pq-p-q-k)-2, \ldots, -(pq-p-q-k) \},\]
there is a Legendrian $L_i\in \mathcal{L}(\K_{(r,s)})$ with 
\[
\tb(L_i)=rs\,\, \text{ and } \rot(L_i)=i.
\]
\item There are two Legendrian knots $K_\pm\in
\mathcal{L}(\K_{(r,s)})$ that do not destabilize but have
\[
\tb(K_\pm)=rs- \left|\frac sr \cdot e_n\right| \,\, \text{ and }  \rot(K_\pm)=\pm r(n-1).
\]
\item All Legendrian knots in $\mathcal{L}(\K_{(r,s)})$ destabilize to
one of the $L_i$ or $K_\pm.$
\item\label{item52} Let $c=(\frac sr \cdot e_n^a- \frac sr\cdot
e_n)-1$.  For any $y\in\N \cup \{0\}$ and $x\leq c$ the Legendrian
$S_\pm^xS_\mp^y(K_\pm)$ is not isotopic to a stabilization of any of
the maximum Thurston-Bennequin invariant Legendrian knots in
$\mathcal{L}(\K_{(r,s)})$ or stabilizations of $K_\mp$.
\item Any two stabilizations of the non-destabilizable
Thurston-Bennequin invariant Legendrian knots in
$\mathcal{L}(\K_{(r,s)})$, except those mentioned in
item~(\ref{item52}), are Legendrian isotopic if they have the same
$\tb$ and $\rot$.
\end{enumerate}
\end{proposition}
\begin{proof}
We follow the standard approach to classifying Legendrian knots in a
given knot type outlined in Section~\ref{sec:pre}.

\medskip

\noindent
{\bf Step I ---} {\em Identify the maximal Thurston-Bennequin
invariant of the knot type and classify Legendrian knots realizing
this:} The computation of the maximal Thurston-Bennequin invariant is
done in Lemma~\ref{lem:maxtb}.

\smallskip

\noindent
{\em  $\bullet$ Construction of maximal Thurston-Bennequin invariant
knots in $\mathcal{L}(\K_{(r,s)})$:} This is identical to part of the
construction in Proposition~\ref{prop:classify1}. Let $N_1$ be a
standard neighborhood of the maximal Thurston-Bennequin invariant
Legendrian $(p,q)$-torus knot.  Inside $N_1$ there are solid tori
corresponding to stabilizing $L$, $(pq-q-p)-k$ times. The range of the
rotation numbers for the Legendrian $(p,q)$-torus knots represented by
these tori is $S=\{(pq-p-q-k), (pq-p-q-k)-2, \ldots,
-(pq-p-q-k)\}$. Denote these tori $S_l$ for $l\in S.$ Inside each
$S_l$ there are two tori $S_l^\pm$ that come from positively or
negatively stabilizing the Legendrian knot corresponding to $S_l$. In
the thickened torus $S_l-S_l^\pm$ there is a unique convex torus
$T_l^\pm$ with dividing slope $\frac sr$. Let $i=sl \pm m$ where
$m=r-sk>0$ is the remainder. Denote by $L_i$ a Legendrian divide on
$T_l^\pm$. We clearly have that $\tb(L_i)=rs$ and the computation in
the proof of Lemma~3.8 in \cite{EtnyreHonda05} gives that
$\rot(L_i)=i$.

\smallskip

\noindent
{\em  $\bullet$ Classification of maximal Thurston-Bennequin invariant
knots in $\mathcal{L}(\K_{(r,s)})$:} If $K\in \mathcal{L}(K_{(r,s)})$
with $\tb(K)=rs$ then $K$ sits on a convex torus with dividing slope
$\frac sr$. Theorem~\ref{thm:partially} and Corollary~\ref{cor:tori}
say that such a torus is one of the ones considered when constructing
the $L_i$. Thus, a by now standard argument, see
\cite{EtnyreHonda01b}, says the torus must be isotopic to one of the
ones used in those constructions from which we can also conclude that
$K$ is isotopic to one of $L_i$.

\medskip

\noindent
{\bf Step II ---} {\em Identify and classify the non-destabilizable, non-maximal
Thurston-Bennequin Legendrian knots in $\mathcal{L}(\K_{(r,s)})$ and
then show the rest destabilize to one of these or a maximal
Thurston-Bennequin Legendrian knot:} Let $N_m^\pm$ be the
non-thickenable solid tori representing $\K$ that were constructed in
Subsection~\ref{ss:nont}.

\smallskip

\noindent
$\bullet$ {\em Constructing the non-destabilizable Legendrian knots:}
Consider the two tori $N_n^\pm$. Let $K_\pm$ be a ruling curve of
slope $(r,s)$ on $\partial N_n^\pm$. It is clear that the twisting of
the contact planes along $K_\pm$ with respect to the framing of
$K_\pm$ coming from $\partial N_n^\pm$ is
\[
-\frac 12\left|K_\pm\cdot \Gamma_{\partial N_n^\pm}\right|= -\left|\frac sr \cdot e_n\right|.
\]
Thus the Thurston-Bennequin invariant (that is the twisting with
respect to the Seifert surface for $K_\pm$) is
\[
\tb(K_\pm)=rs-\left|\frac sr \cdot e_n\right|.
\]
Just as in Step I of the proof of Proposition~\ref{prop:classify1} we
compute
\[
\rot(K_\pm)=\pm r(n-1).
\]

\smallskip

\noindent
$\bullet$ {\em Proving all non-maximal Thurston-Bennequin invariant
knots either destabilize or have $\tb=rs-|\frac sr \cdot e_n|$ and sit
as a ruling curve on $\partial N_n^\pm$:} Let $L$ be a Legendrian knot
in $\L(\K_{(r,s)})$ with $\tb(L)<rs$. Let $S$ be a solid torus
representing the knot type $\K$ that contains $L$ in its boundary. We
know that the twisting of the contact planes with respect to $\partial
S$ is negative so we can make $\partial S$ convex without moving
$L$. If $L$ does not intersect the dividing curves $\Gamma_{\partial
S}$ minimally (for curves in their homology classes) then we will see
a bypass for $L$ on $\partial S$ and hence $L$ destabilizes. So we can
assume that $L$ intersects $\Gamma_{\partial S}$ minimally.

Now if the dividing slope $t$ of $\partial S$ is not $e_n$ then there
are three cases to consider. Case one is when $t\not\in [e_m, e_m^a)$
for any $m$. In this case Theorem~\ref{thm:partially} tells us that
$S$ can be thickened to a standard neighborhood of a maximal
Thurston-Bennequin knot in $\mathcal{L}(\K)$. Thus there is a convex
torus $T$ parallel to $\partial S$ (either inside $S$ or outside $S$
depending on $t$) with dividing slope $\frac sr.$ We can use an
annulus between $T$ and $\partial S$ with boundary on $L$ and a
Legendrian divide on $T$ to find a bypass for $L$ and hence $L$
destabilizes. Case two is when $t\in [e_m,e_m^a)$ for $m\not=n.$
Lemma~\ref{lem:basicrange} says that $|t\cdot \frac sr|$ is strictly
greater than $|\frac sr \cdot e_m^a|$ and $|\frac sr \cdot e_m^c|$
(since $t$ is on the interior of $[e_m^c, e_m^a]$). Thus there is a
torus $T$ in $S$ with dividing slope $e_m^a$. Using an annulus between
$K$ on $T$ and a $\frac sr$ ruling curve on $T$ we find a bypass for
$L$ and hence a destabilization. Finally in case three we consider
$t\in (e_n, e_n^a)$. In this case we have that $|\frac sr \cdot t| >
|\frac sr \cdot e_n|$.  We can thus use an annulus between $L$ on
$\partial S$ and a $\frac sr$ ruling on $\partial N_n^\pm$ to find a
bypass for $L$.

If $t=e_n$ then $L$ is a ruling curve on $\partial S$. If $S$ is not
$N_n^\pm$ then $S$ will thicken to $N_1$ and thus we can again
destabilize $L$ as in case one of the previous paragraph. So we see
that $L$ will destabilize unless it is a ruling curve on $N_n^\pm$. Of
course in this case $\tb(L)=rs-|\frac sr \cdot e_n|$.

\smallskip

\noindent
$\bullet$ {\em Proving the knots $K_\pm$ do not destabilize:} It
$K_\pm$ destabilized then by the above work they would be
stabilizations of one of the $L_i$. Thus $K_\pm$ could be put on some
convex torus other than $\partial N_n^\pm$, but this contradicts
Proposition~\ref{curvesonpt2}.

\smallskip

\noindent
$\bullet$ {\em Proving any Legendrian knots with $\tb=rs-|\frac sr
\cdot e_n|$ either destabilize or are isotopic to $K_\pm$:} This is
immediate from the work above and Corollary~\ref{cor:tori}.

\medskip

\noindent
{\bf Step III ---} {\em Determine which stabilizations of the $K_\pm$
and $L_i$ are Legendrian isotopic:} The stabilizations of the $L_i$
all become Legendrian isotopic whenever they have the same
Thurston-Bennequin invariants as discussed in Step III of the proof of
Proposition~\ref{prop:classify1}.

From Proposition~\ref{curvesonpt2} we know that
$S_\pm^xS_\mp^y(K_\pm)$, for any $y\in\N\cup\{0\}$ and $x\leq c$, can
be put only on the convex torus $\partial N_n^\pm$. Thus it is clear
that $S_\pm^xS_\mp^y(K_\pm)$ is not isotopic to any stabilization of a
$L_i$ or of $K_\mp$.

We also know from Propositon~\ref{curvesonpt1} that
$S_\pm^{c+1}S_\mp^y(K_\pm)$ can be put on a convex torus that bounds a
solid torus that thickens to $N_1$ and thus is a stabilization of the
$L_i$.
\end{proof}

\begin{proof}[Proof of Theorem~\ref{qual3} and Theorem~\ref{main} ]
Theorem~\ref{main} is an immediate consequence of
Propositions~\ref{prop:classify1} and~\ref{prop:classify2} together
with Theorems~\ref{poscableclass}, \ref{negcableclass}
and~\ref{intermediatecableclass}. Theorem~\ref{qual3} is clear from
the statement of Theorem~\ref{main}.
\end{proof}

\begin{proof}[Proof of Theorem~\ref{qual4} and Theorem~\ref{maint}]
Theorem~\ref{tequivaletsl} tells us that the classification of
transverse knots is equivalent to the classification of Legendrian
knots up to negative stabilization. Thus the Theorem~\ref{maint} is a
corollary of Theorem~\ref{main}. Theorem~\ref{qual4} follows from
Theorem~\ref{maint} once one observes that that if we choose $\frac sr
= me_k + n e_k^a$ (where the addition is done as on the Farey
tessellation), then $\frac sr \cdot \frac{1}{pq-p-q} > n$.  As a
result, the non-destabilizable transverse knot will have self-linking
number at least $2n$ less than maximal; furthermore, it will take
$\frac sr \cdot e_k^a = m$ stabilizations before it becomes isotopic
to a stabilization of the maximal self-linking number transverse knot.
\end{proof}

\section{Cables of the trefoil}
\label{non-simple cables trefoil}
We are now ready to classify non-Legendrian simple cables of the
positive trefoil knot.
\begin{proposition}\label{class23}
Let $\K$ be the positive trefoil knot. Suppose that $n\geq 1$ and $(r,s)$ 
is a pair of relatively prime integers such that $\frac sr\in [n,n+1)$. Then 
the $(r,s)$-cable of $\K$, $\K_{(r,s)}$, is not Legendrian simple and
Legendrian knots in this knot type have the following classification.
\begin{enumerate}
\item The maximal Thurston-Bennequin number is
$\overline{\tb}(\K_{(r,s)})=rs$.
\item There are $n$ Legendrian knots $L_\pm^j\in
\mathcal{L}(\K_{(r,s)}), j=1,\ldots, n$, with
\[
\tb(L_\pm^j)=rs\,\, \text{ and } \rot(L_\pm^j)=\pm(s-r).
\]
\item If $\frac sr \not=n$ then there are two Legendrian knots
$K_\pm\in \mathcal{L}(\K_{(r,s)})$ that do not destabilize but have
\[
\tb(K_\pm)=rs- \left|\frac sr \cdot (n+1)\right| \,\, \text{ and }  \rot(K_\pm)=\pm \left(s-r+ \left|\frac sr \cdot (n+1)\right|\right).
\]
\item All Legendrian knots in $\mathcal{L}(\K_{(r,s)})$ destabilize to
one of the $L^j_\pm$ or $K_\pm$.
\item\label{item511} Let $c=r-1$.  For any $y\in\N$ and $x\leq c$ the
Legendrian $S_\pm^xS_\mp^y(L^j_\pm)$ is not isotopic to a
stabilization of any of the other $L_\pm^{j}$'s the $L_\mp^j$, $K_\pm$
or $K_\mp$.
\item\label{item512} Let $c'=r-\left|\frac sr \cdot
(n+1)\right|-1$. For any $y\in\N\cup\{0\}$ and $x\leq c'$ the
Legendrian $S_\pm^xS_\mp^y(K_\pm)$ is not isotopic to a stabilization
of any of the $L_\pm^{j}$'s or $K_\mp$.
\item Any two stabilizations of the $L_\pm^j$ or $K_\pm$, except those
mentioned in item~(\ref{item511}) and~(\ref{item512}), are Legendrian
isotopic if they have the same $\tb$ and $\rot$.
\end{enumerate}
\end{proposition}
\begin{proof}
We follow the standard approach to classifying Legendrian knots used
above.

\medskip

\noindent
{\bf Step I ---} {\em Identify the maximal Thurston-Bennequin
invariant of the knot type and classify Legendrian knots realizing
this:} The computation of the maximal Thurston-Bennequin invariant is
done in Lemma~\ref{lem:maxtb}.

\smallskip

\noindent
{\em $\bullet$ Construction of maximal Thurston-Bennequin invariant
knots in $\mathcal{L}(\K_{(r,s)})$:} This is identical to the
construction from the last section. Let $N_1$ be a standard
neighborhood of the maximal Thurston-Bennequin invariant Legendrian
positive trefoil knot.  Inside $N_1$ there are two solid tori $S^\pm$
that come from positively or negatively stabilizing the Legendrian
knot corresponding to $N_1$. In the thickened torus $N_1-S^\pm$ there
is a unique convex torus $T^\pm$ with dividing slope $\frac sr$. Let
$L^1_\pm$ be a Legendrian divide on $T^\pm$.  We clearly have that
$\tb(L_i)=rs$ and the computation in the proof of Lemma~3.8 in
\cite{EtnyreHonda05} gives that $\rot(L_i)=\pm(s-r)$.

Now consider the non-thickenable tori $N_j^\pm$ from
Theorem~\ref{thm:partially}. For $j\leq n$ we can find a convex torus
$T_j^\pm$ with dividing slope $\frac sr.$ Let $L^j_\pm$ be a
Legendrian divide on $T_j^\pm$. Again it is clear that
$\tb(K_\pm)=rs.$ Recall that from Lemma~\ref{rcomp} we know that
\[
\rot(L^j_\pm)=r\rot(\partial D) + s\rot(\partial \Sigma)
\]
where $D$ is a meridional disk for $T_j^\pm$ with Legendrian boundary
and $\Sigma$ is a surface outside the solid torus $T_j^\pm$ bounds
with Legendrian boundary on $T_j^\pm.$ If $D'$ and $\Sigma'$ are the
corresponding surfaces for $\partial N_j^\pm$ then we know from
Lemma~\ref{candidates} that $\rot (\partial D')=\pm(j-1)$ and
$\rot(\partial \Sigma')=0.$ Thus the rotation number of an
$(r,s)$-ruling curve on $\partial N_n^\pm$ is $\pm r(j-1)$. To compute
the rotation number for the Legendrian divide on $T_j^\pm$ we use the
classification of tight contact structures on thickened tori, as given
in \cite{Honda00a}, and the fact that $N_j^\pm$ is universally
tight. In particular, we can compute the relative Euler class $e$ of
the thicken torus cobranded by $N_j^\pm$ and $T^\pm$:
\[
P.D.(e)=\pm( (r,s) -(1, j)) \in H_1(T^2\times I;\Z),
\]
where $P.D.$ stands for the Poincar\'e Dual and we are using the basis
for $H_1$ given by the meridian and longitude.  We can use this to
compute the difference between the rotation number of the $(r,s)$
curve on $\partial N_j^\pm$ and on $T_j^\pm$ which is
$\pm(r(s-j)-s(r-1)$. Thus we have that $\rot(L^j_\pm)=\pm (s-r)$.

\smallskip

\noindent
{\em $\bullet$ Classification of maximal Thurston-Bennequin invariant
knots in $\mathcal{L}(\K_{(r,s)})$:} If $K\in \mathcal{L}(K_{(r,s)})$
with $\tb(K)=rs$ then $K$ sits on a convex torus with dividing slope
$\frac sr$. Theorem~\ref{thm:partially} and Corollary~\ref{cor:tori}
say that such a torus is one of the ones considered when constructing
the $L^j_\pm$. Thus, a by now standard argument, see
\cite{EtnyreHonda01b}, says the torus must be isotopic to one of the
ones used in those constructions from which we can also conclude that
$K$ is isotopic to one of the $L^j_\pm$.

\medskip

\noindent
{\bf Step II ---} {\em Identify and classify the non-destabilizable, non-maximal
Thurston-Bennequin Legendrian knots in $\mathcal{L}(\K_{(r,s)})$ and
then show the rest destabilize to one of these or a maximal
Thurston-Bennequin Legendrian knot:} Let $N_m^\pm$ be the
non-thickenable solid tori representing $\K$ that were constructed in
Subsection~\ref{ss:nont}.

\smallskip

\noindent
$\bullet$ {\em Constructing the non-destabilizable Legendrian knots:}
If $\frac sr=n$ then there are no non-destabilizable knots. Otherwise
consider the two tori $N_{n+1}^\pm$. Let $K_\pm$ be a ruling curve of
slope $(r,s)$ on $\partial N_{n+1}^\pm$. It is clear that the twisting
of the contact planes along $K_\pm$ with respect to the framing of
$K_\pm$ coming from $\partial N_n^\pm$ is
\[
-\frac 12\left|K_\pm\cdot \Gamma_{\partial N_{n+1}^\pm}\right|= -\left|\frac sr \cdot (n+1)\right|.
\]
Thus the Thurston-Bennequin invariant (that is the twisting with
respect to the Seifert surface for $K_\pm$) is
\[
\tb(K_\pm)=rs-\left|\frac sr \cdot (n+1)\right|.
\]
Just as in the proof of Proposition~\ref{prop:classify1} we compute
\[
\rot(K_\pm)=\pm \left(s-r-\left|\frac sr \cdot (n+1)\right|\right),
\]
or, more simply, $\rot(K_\pm) = \pm rn$.

\noindent
$\bullet$ {\em Proving all non-maximal Thurston-Bennequin invariant
knots either destabilize or have $\tb=rs-|\frac sr \cdot (n+1)|$ and
sit as a ruling curve on $\partial N_{n+1}^\pm$:} Assume that $\frac
sr\not=n$ (since otherwise there are no-non-destabilizable knots). Let
$L$ be a Legendrian knot in $\L(\K_{(r,s)})$ with $\tb(L)<rs$. Let $S$
be a solid torus representing the knot type $\K$ that contains $L$ in
its boundary. We know that the twisting of the contact planes with
respect to $\partial S$ is negative so we can make $\partial S$ convex
without moving $L$. If $L$ does not intersect the dividing curves
$\Gamma_{\partial S}$ minimally (for curves in their homology classes)
then we will see a bypass for $L$ on $\partial S$ and hence $L$
destabilizes. So we can assume that $L$ intersect $\Gamma_{\partial
S}$ minimally.

Now if the dividing slope $t$ of $\partial S$ is not $n+1$ then there
are three cases to consider. If $t<0$ then $S$ thickens to $N_1$ and
in particular there is a convex torus with dividing slope $\frac sr$
either inside or outside $S$. We may use an annulus between $L$ and a
dividing curve on this torus to destabilize $L$. If $t>n+1$, then $S$
contains a solid torus $S'$ with convex boundary having infinite
dividing slope. Lemma~\ref{lem:basicrange} guarantees that $|t\cdot
\frac sr|$ is greater than $|\frac 10\cdot \frac sr|$. Thus we may
take a convex annulus from $L$ to a ruling curve on $\partial S'$ and
use the Imbalance Principle to find a bypass, and hence a
destabilization, for $L$. Finally if $t\in(n,n+1)$, then there is a
torus with dividing slope $\frac sr$ either inside or outside of $S$,
and we may use an annulus between $L$ and a dividing curve on this
torus to destabilize $L$.

If $t=n+1$ then $L$ is a ruling curve on $\partial S$. If $S$ is not
$N_{n+1}^\pm$ then $S$ will thicken to $N_k$ for some $k\geq n$ and
thus we can again destabilize $L$ as in case one of the previous
paragraph. So we see that $L$ will destabilize unless it is a ruling
curve on $N_{n+1}^\pm$. Of course in this case $\tb(L)=rs-|\frac sr
\cdot (n+1)|$.

\smallskip

\noindent
$\bullet$ {\em Proving the knots $K_\pm$ do not destabilize:} If
$K_\pm$ destabilized then by the above work they would be
stabilizations of one of the $L^j_\pm$. Thus $K_\pm$ could be put on
some convex torus other than $\partial N_n^\pm$, but this contradicts
Proposition~\ref{curvesonpt2}.

\smallskip

\noindent
$\bullet$ {\em Proving any Legendrian knots with $\tb=rs-|\frac sr
\cdot (n+1)|$ either destabilize or are isotopic to $K_\pm$:} This is
immediate from the work above and Corollary~\ref{cor:tori}.

\medskip

\noindent
{\bf Step III ---} {\em Determine which stabilizations of the $K_\pm$
and $L_i$ are Legendrian isotopic:} The stabilizations of the $L_i$
are shown to be Legendrian isotopic when they have the same classical
invariants in the usual fashion as discussed in the proof of
Proposition~\ref{prop:classify1}.

From Proposition~\ref{curvesonpt1} we know that
$S_\pm^xS_\mp^y(L^j_\pm)$, for any $y\in\N \cup \{0\}$ and $x\leq c$,
can be put only on the convex torus $\partial N_j^\pm$. Thus it is
clear that $S_\pm^xS_\mp^y(K_\pm)$ is not isotopic to any
stabilization of any of the other $L^j_\pm$, $K_\pm$ or
$K_\mp$. Similarly if $\frac sr\not=n$ then for $x\leq c'$,
Proposition~\ref{curvesonpt2} says that $S_\pm^xS_\mp^y(K_\pm)$ can
only be put on the convex torus $\partial N_{n+1}^\pm$ and hence is
not isotopic to any stabilization of the $L^j_\pm$ or to $K_\mp$.

We also know from Propositons~\ref{curvesonpt1} and ~\ref{curvesonpt2}
that $S_\pm^{c+1}S_\mp^y(L^j_\pm)$ and $S_\pm^{c'+1}S_\mp^y(K_\pm)$
can be put on a convex torus that bounds a solid torus that thickens
to $N_1$ and thus is a stabilization of the $L^1_\pm$.
\end{proof}

\begin{proof}[Proof of Theorem~\ref{qual1} and Theorem~\ref{main23}]
Theorem~\ref{main23} simply collects the results from
Proposition~\ref{class23} and Theorems~\ref{poscableclass}
and~\ref{negcableclass}. For Theorem~\ref{qual1} we can choose $\frac
sr= \frac{kn+m(n-1)}{k+m}$. One may easily check using
Theorem~\ref{main23} that $\mathcal{L}(\K_{(r,s)})$ contains $n-1$
Legendrian knots $L_1,\ldots, L_{n-1}$ with maximal Thurston-Bennequin
invariant (which will be $rs$ in this case) and rotation number
$s-r$. It also contains one non-destabilizable knot $L'$ with $\tb=
rs-|\frac sr\cdot n|=rs-m$ and rotation number $s-r+m$. Moreover, one
must stabilize $L'$ positively $k$ times before it becomes isotopic to
a stabilization of one of the $L_i$.
\end{proof}

\begin{proof}[Proof of Theorem~\ref{qual2} and Theorem~\ref{main23t}]
Theorem~\ref{tequivaletsl} tells us that the classification of
transverse knots is equivalent to the classification of Legendrian
knots up to negative stabilization. Thus the Theorem~\ref{main23t} is
a corollary of Theorem~\ref{main23}. Turning to Theorem~\ref{qual2} we
see that choices similar to those in the previous proof yield the
desired result.
\end{proof}

\def\cprime{$'$} \def\cprime{$'$}

\end{document}